\documentclass[a4paper,twoside]{article}
\usepackage{a4}
\usepackage{amssymb}
\usepackage{amsmath}
\usepackage{upref}
\usepackage[active]{srcltx}
\usepackage[dvips,colorlinks,citecolor=blue,linkcolor=blue]{hyperref}
\usepackage[dvipsnames]{color}
\allowdisplaybreaks[2] 
\newcount\minutes \newcount\hours
\hours=\time
\divide\hours 60
\minutes=\hours
\multiply\minutes -60
\advance\minutes \time
\newcommand{\klockan}{\the\hours:{\ifnum\minutes<10 0\fi}\the\minutes}
\newcommand{\tid}{\today\ \klockan}
\newcommand{\prtid}{\smash{\raise 10mm \hbox{\LaTeX ed \tid}}}
\renewcommand{\prtid}{}
\makeatletter
\def\sectionmark#1{} 
\def\subsectionmark#1{}
\newcommand{\sectnr}{\ifnum \c@secnumdepth >\z@
                 \thesection.\hskip 1em\relax \fi}
\def\@evenhead{\footnotesize\rm\thepage\hfil\leftmark\hfil\llap{\prtid}}
\def\@oddhead{\footnotesize\rm\rlap{\prtid}\hfil\rightmark\hfil\thepage}
\def\tableofcontents{\section*{Contents} 
 \@starttoc{toc}}
\makeatother
\makeatletter
\def\@biblabel#1{#1.}
\makeatother
\makeatletter
\let\Thebibliography=\thebibliography
\renewcommand{\thebibliography}[1]{\def\@mkboth##1##2{}\Thebibliography{#1}
\addcontentsline{toc}{section}{References}
\frenchspacing 
\setlength{\@topsep}{0pt}
\setlength{\itemsep}{0pt}%
\setlength{\parskip}{0pt plus 2pt}%
}
\makeatother
\makeatletter
\def\mdots@{\mathinner.\nonscript\!.%
 \ifx\next,.\else\ifx\next;.\else\ifx\next..\else
 \nonscript\!\mathinner.\fi\fi\fi}
\let\ldots\mdots@
\let\cdots\mdots@
\let\dotso\mdots@
\let\dotsb\mdots@
\let\dotsm\mdots@
\let\dotsc\mdots@
\def\vdots{\vbox{\baselineskip2.8\p@ \lineskiplimit\z@
    \kern6\p@\hbox{.}\hbox{.}\hbox{.}\kern3\p@}}
\def\ddots{\mathinner{\mkern1mu\raise8.6\p@\vbox{\kern7\p@\hbox{.}}%
    \raise5.8\p@\hbox{.}\raise3\p@\hbox{.}\mkern1mu}}
\makeatother
\makeatletter
\let\Enumerate=\enumerate
\renewcommand{\enumerate}{\Enumerate%
\setlength{\@topsep}{0pt}
\setlength{\itemsep}{0pt}%
\setlength{\parskip}{0pt plus 1pt}%
\renewcommand{\theenumi}{\textup{(\alph{enumi})}}%
\renewcommand{\labelenumi}{\theenumi}%
}
\let\endEnumerate=\endenumerate
\renewcommand{\endenumerate}{\endEnumerate\unskip}
\makeatother
\makeatletter
\def\@seccntformat#1{\csname the#1\endcsname.\quad}
\makeatother
\newcommand{\authortitle}[2]{\author{#1}\title{#2}\markboth{#1}{#2}}
\newcommand{\auth}[2]{{#2. #1}}
\def\idxauth{\auth}
\newcommand{\art}[6]{{\sc #1, \rm #2, \it #3\/ \bf #4 \rm (#5), \mbox{#6}.}}
\newcommand{\artnopt}[6]{{\sc #1, \rm #2, \it #3\/ \bf #4 \rm (#5), \mbox{#6}}}
\newcommand{\artprep}[3]{{\sc #1, \rm #2, \it #3.}}
\newcommand{\artin}[3]{{\sc #1, \rm #2,  in #3.}}
\newcommand{\arttoappear}[3]{{\sc #1, \rm #2, to appear in \it #3}}
\newcommand{\book}[3]{{\sc #1, \it #2, \rm #3.}}
\newcommand{\AND}{{\rm and }}
\RequirePackage{amsthm}
\newtheoremstyle{descriptive}%
  {\topsep}   
  {\topsep}   
  {\rmfamily} 
  {}          
  {\bfseries} 
  {.}         
  { }         
  {}          
\newtheoremstyle{propositional}%
  {\topsep}   
  {\topsep}   
  {\itshape}  
  {}          
  {\bfseries} 
  {.}         
  { }         
  {}          
\theoremstyle{propositional}
\newtheorem{thm}{Theorem}[section]
\newtheorem{prop}[thm]{Proposition}
\newtheorem{lem}[thm]{Lemma}
\theoremstyle{descriptive}
\newtheorem{deff}[thm]{Definition}
\newtheorem{example}[thm]{Example}
\newtheorem{remark}[thm]{Remark}
\makeatletter
\renewenvironment{proof}[1][\proofname]{\par
  \pushQED{\qed}%
  \normalfont
  \trivlist
  \item[\hskip\labelsep
        \itshape
    #1\@addpunct{.}]\ignorespaces
}{%
  \popQED\endtrivlist\@endpefalse
}
\makeatother
\def\cprime{{\mathsurround0pt$'$}}
\newcommand{\setm}{\setminus}
\renewcommand{\emptyset}{\varnothing}
\def\vint{\mathop{\mathchoice%
          {\setbox0\hbox{$\displaystyle\intop$}\kern 0.22\wd0%
           \vcenter{\hrule width 0.6\wd0}\kern -0.82\wd0}%
          {\setbox0\hbox{$\textstyle\intop$}\kern 0.2\wd0%
           \vcenter{\hrule width 0.6\wd0}\kern -0.8\wd0}%
          {\setbox0\hbox{$\scriptstyle\intop$}\kern 0.2\wd0%
           \vcenter{\hrule width 0.6\wd0}\kern -0.8\wd0}%
          {\setbox0\hbox{$\scriptscriptstyle\intop$}\kern 0.2\wd0%
           \vcenter{\hrule width 0.6\wd0}\kern -0.8\wd0}}%
          \mathopen{}\int}
\newcommand{\Cp}{{C_p}}
\newcommand{\CpB}{{C_p^B}}
\DeclareMathOperator{\diam}{diam}
\DeclareMathOperator{\Div}{div}
\DeclareMathOperator{\capp}{cap}
\newcommand{\cp}{\capp_p}
\newcommand{\cpt}{\capp_{\pt}}
\newcommand{\cq}{\capp_q}
\newcommand{\ctwo}{\capp_2}
\newcommand{\cone}{\capp_1}
\newcommand{\grad}{\nabla}
\DeclareMathOperator{\dist}{dist}
\DeclareMathOperator{\interior}{int}
\DeclareMathOperator*{\essliminf}{ess\,lim\,inf}
\DeclareMathOperator*{\esslimsup}{ess\,lim\,sup}
\DeclareMathOperator*{\esssup}{ess\,sup}
\newcommand{\bdry}{\partial}
\newcommand{\loc}{_{\rm loc}}
\newcommand{\simge}{\gtrsim}
\newcommand{\simle}{\lesssim}
{\catcode`p =12 \catcode`t =12 \gdef\eeaa#1pt{#1}}      
\def\accentadjtext#1{\setbox0\hbox{$#1$}\kern   
                \expandafter\eeaa\the\fontdimen1\textfont1 \ht0 }
\def\accentadjscript#1{\setbox0\hbox{$#1$}\kern 
                \expandafter\eeaa\the\fontdimen1\scriptfont1 \ht0 }
\def\accentadjscriptscript#1{\setbox0\hbox{$#1$}\kern   
                \expandafter\eeaa\the\fontdimen1\scriptscriptfont1 \ht0 }
\def\accentadjtextback#1{\setbox0\hbox{$#1$}\kern       
                -\expandafter\eeaa\the\fontdimen1\textfont1 \ht0 }
\def\accentadjscriptback#1{\setbox0\hbox{$#1$}\kern     
                -\expandafter\eeaa\the\fontdimen1\scriptfont1 \ht0 }
\def\accentadjscriptscriptback#1{\setbox0\hbox{$#1$}\kern 
                -\expandafter\eeaa\the\fontdimen1\scriptscriptfont1 \ht0 }
\def\itoverline#1{{\mathsurround0pt\mathchoice
        {\rlap{$\accentadjtext{\displaystyle #1}
                \accentadjtext{\vrule height1.593pt}
                \overline{\phantom{\displaystyle #1}
                \accentadjtextback{\displaystyle #1}}$}{#1}}
        {\rlap{$\accentadjtext{\textstyle #1}
                \accentadjtext{\vrule height1.593pt}
                \overline{\phantom{\textstyle #1}
                \accentadjtextback{\textstyle #1}}$}{#1}}
        {\rlap{$\accentadjscript{\scriptstyle #1}
                \accentadjscript{\vrule height1.593pt}
                \overline{\phantom{\scriptstyle #1}
                \accentadjscriptback{\scriptstyle #1}}$}{#1}}
        {\rlap{$\accentadjscriptscript{\scriptscriptstyle #1}
                \accentadjscriptscript{\vrule height1.593pt}
                \overline{\phantom{\scriptscriptstyle #1}
                \accentadjscriptscriptback{\scriptscriptstyle #1}}$}{#1}}}}
\def\itunderline#1{{\mathsurround0pt\mathchoice
        {\rlap{$\underline{\phantom{\displaystyle #1}
                \accentadjtextback{\displaystyle #1}}$}{#1}}
        {\rlap{$\underline{\phantom{\textstyle #1}
                \accentadjtextback{\textstyle #1}}$}{#1}}
        {\rlap{$\underline{\phantom{\scriptstyle #1}
                \accentadjscriptback{\scriptstyle #1}}$}{#1}}
        {\rlap{$\underline{\phantom{\scriptscriptstyle #1}
                \accentadjscriptscriptback{\scriptscriptstyle #1}}$}{#1}}}}
\newcommand{\al}{\alpha}
\newcommand{\alp}{\alpha}
\newcommand{\be}{\beta}
\newcommand{\ga}{\gamma}
\newcommand{\dmu}{d\mu}
\newcommand{\de}{\delta}
\newcommand{\eps}{\varepsilon}
\newcommand{\la}{\lambda}
\newcommand{\s}{\sigma}
\newcommand{\om}{\omega}
\newcommand{\Om}{\Omega}
\newcommand{\z}{\zeta}
\newcommand{\p}{{$p\mspace{1mu}$}}
\newcommand{\R}{\mathbf{R}}
\newcommand{\Sp}{\mathbf{S}}

\newcommand{\Z}{\mathbf{Z}}
\newcommand{\N}{\mathbf{N}}
\newcommand{\Rt}{\widetilde{R}}
\newcommand{\limplus}{{\mathchoice{\raise.17ex\hbox{$\scriptstyle +$}}
                {\raise.17ex\hbox{$\scriptstyle +$}}
                {\raise.1ex\hbox{$\scriptscriptstyle +$}}
                {\scriptscriptstyle +}}}
\newcommand{\Np}{N^{1,p}}
\newcommand{\Nploc}{N^{1,p}\loc}
\newcommand{\pt}{{\tilde{p}}}
\newcommand{\gt}{\tilde{g}}
\newcommand{\rt}{\tilde{r}}
\newcommand{\lQo}{\itunderline{Q}_0}
\newcommand{\uQo}{\itoverline{Q}_0}
\newcommand{\lSo}{\itunderline{S}_0}
\newcommand{\uSo}{\itoverline{S}_0}
\newcommand{\lQi}{\itunderline{Q}_\infty}
\newcommand{\uQi}{\itoverline{Q}_\infty}
\newcommand{\lSi}{\itunderline{S}_\infty}
\newcommand{\uSi}{\itoverline{S}_\infty}
\newcommand{\lQ}{\itunderline{Q}}
\newcommand{\uQ}{\itoverline{Q}}

\newcommand{\loq}{{\underline{q}}}
\newcommand{\uq}{\overline{q}}
\newcommand{\po}{{p_0}}
\newcommand{\ro}{{\rho}}
\newcommand{\Ga}{\Gamma}
\newcommand{\Lp}{L^p}

\makeatletter
\newcommand{\setcurrentlabel}[1]{\def\@currentlabel{#1}}
\makeatother

\numberwithin{equation}{section}
\newcommand{\imp}{\mathchoice{\quad \Longrightarrow \quad}{\Rightarrow}
                {\Rightarrow}{\Rightarrow}}
\newenvironment{ack}{\medskip{\it Acknowledgement.}}{}

\begin{document}

\authortitle{Anders Bj\"orn, Jana Bj\"orn
    and Juha Lehrb\"ack}
{Sharp capacity estimates for annuli in weighted $\R^n$ and in metric spaces}
\author{
Anders Bj\"orn \\
\it\small Department of Mathematics, Link\"opings universitet, \\
\it\small SE-581 83 Link\"oping, Sweden\/{\rm ;}
\it \small anders.bjorn@liu.se
\\
\\
Jana Bj\"orn \\
\it\small Department of Mathematics, Link\"opings universitet, \\
\it\small SE-581 83 Link\"oping, Sweden\/{\rm ;}
\it \small jana.bjorn@liu.se
\\
\\
Juha Lehrb\"ack \\
\it\small Department of Mathematics and Statistics, University of Jyv\"askyl\"a,\\
\it\small P.O. Box 35\/ \textup{(}MaD\/\textup{)}, FI-40014 University of Jyv\"askyl\"a, Finland\/{\rm ;}
\it \small juha.lehrback@jyu.fi
\\
}

\date{}

\maketitle

\noindent{\small
 {\bf Abstract}.
We obtain estimates for the nonlinear variational
capacity of annuli in weighted $\R^n$ and in
metric spaces. 
We introduce four different (pointwise)
exponent sets,
show that
they all play fundamental roles for capacity 
estimates, and
also demonstrate that whether an end point of an exponent
set is attained or not is important.
As a consequence of our estimates we obtain, 
for instance, criteria for
points to have zero (resp.\ positive) capacity.
Our discussion holds in rather general metric 
spaces, including Carnot groups and many manifolds, but it is just
as relevant on weighted $\R^n$.
Indeed, to illustrate the sharpness of 
our estimates, we give several
examples of radially weighted $\R^n$,
which are  
based on 
quasiconformality of radial 
stretchings in $\R^n$.
}

\bigskip

\noindent {\small \emph{Key words and phrases}:
annulus, capacity,  doubling measure,  exponent sets, Jacobian,  
metric space, 
Newtonian space, \p-admissible weight, Poincar\'e inequality, 
quasiconformal mapping, 
radial stretching, radial weight, 
reverse-doubling, sharp estimate, Sobolev capacity,
Sobolev space, upper gradient, variational capacity,
weighted $\R^n$.
}

\medskip

\noindent {\small Mathematics Subject Classification (2010):
Primary: 31C45; Secondary: 30C65, 30L99, 31B15, 31C15, 31E05.
}

\section{Introduction}

Our aim in this paper is to give sharp estimates for 
the variational \p-capacity of annuli in metric spaces. 
Such estimates play an important role for instance in the study of
singular solutions and Green functions for (quasi)linear equations 
in (weighted) Euclidean spaces and
in more general settings, such as subelliptic equations associated
with vector fields and on Heisenberg groups, 
see e.g.\ Serrin~\cite{Serrin}, 
Capogna--Danielli--Garofalo~\cite{CaDaGa2},
and Danielli--Garofalo--Marola~\cite{DaGaMa}
for discussion and applications.
Recall that analysis and nonlinear
potential theory (including capacities) 
have during the last two decades been developed on very general metric
spaces, including compact Riemannian manifolds and their Gromov--Hausdorff
limits, and Carnot--Carath\'eodory spaces.

Sharp capacity estimates depend in a crucial way on good bounds
for the (relative) measures of balls. 
For instance, recall that for $0<2r\le R$, the 
variational \p-capacity 
$\cp(B(x,r),B(x,R))$ of the annulus $B(x,R)\setm B(x,r)$
in (unweighted) $\R^n$ is  
comparable to $r^{n-p}$ if $p<n$ and to $R^{n-p}$ if $p>n$, see
e.g.\ Example~2.12 in Heinonen--Kilpel\"ainen--Martio~\cite{HeKiMa}.
In both cases, $r^n$ and $R^n$ are comparable to the Lebesgue measure of one of the balls
defining the annulus. 
For $p=n$, the \p-capacity contains a logarithmic term of the ratio $R/r$.
Thus, the dimension $n$ (or rather the way in which the Lebesgue measure scales on
balls with different radii) determines (together with \p) 
the form of the estimates for
the \p-capacity of annuli.

If $X=(X,d,\mu)$ is a metric space equipped with a doubling measure $\mu$ 
(i.e.\ $\mu(2B)\le C \mu(B)$ for all balls $B\subset X$), then an iteration of the
doubling condition 
shows that there exist $q>0$ and $C>0$ such that
\begin{equation*}
\frac{\mu(B(x,r))}{\mu(B(x,R))}\ge C\Bigl(\frac rR\Bigr)^q
\end{equation*}
for all $x\in X$ and $0<r<R$. In addition, a converse estimate, with some
exponent $0<q'\le q$, holds under the assumption 
that $X$ is connected
(see Section~\ref{sect-sets-of-exponents} for details).
Motivated by these observations, we  
introduce the following \emph{exponent sets} for $x\in X$:
\begin{align*}
  \lQo(x)
        & :=\Bigr\{q>0 : \text{there is $C_q$ so that } 
        \frac{\mu(B(x,r))}{\mu(B(x,R))}  \le C_q \Bigl(\frac{r}{R}\Bigr)^q 
        \text{ for } 0 < r < R \le 1
        \Bigl\}, \\
  \lSo(x)
          & :=\{q>0 : \text{there is $C_q$ so that } 
        \mu(B(x,r))  \le C_q r^q 
        \text{ for } 0 < r  \le 1
        \}, \\
  \uSo(x)
          & :=\{q>0 : \text{there is $C_q$ so that } 
        \mu(B(x,r))  \ge C_q r^q 
        \text{ for } 0 < r  \le 1
        \}, \\
  \uQo(x)
        & :=\Bigr\{q>0 : \text{there is $C_q$ so that } 
        \frac{\mu(B(x,r))}{\mu(B(x,R))}  \ge C_q \Bigl(\frac{r}{R}\Bigr)^q 
        \text{ for } 0 < r < R \le 1
        \Bigl\}.
\end{align*}
Here the subscript $0$ refers to the fact that only small radii are 
considered; we 
shall later define similar exponent sets with large radii as well.
In general, all of these 
sets can be different, as shown in 
Examples~\ref{ex1} and~\ref{ex-abcd-alt}. 

The above exponent sets turn out to be of fundamental importance for distinguishing
between the cases in which the sharp estimates for capacities are different,
in a similar way as the dimension in $\R^n$ does. 
Let us mention here that
Garofalo--Marola~\cite{GaMa} defined a pointwise dimension  
$q(x)$ (called $Q(x)$ therein) 
and established certain capacity estimates for the cases
$p<q(x)$, $p=q(x)$  and $p>q(x)$.  
In our terminology their $q(x)=\sup\lQ(x)$,
where $\lQ(x)$ is a global version of $\lQo(x)$, see 
Section~\ref{sect-sets-of-exponents}.
However, it turns out that the situation is in fact even more subtle than 
indicated in~\cite{GaMa}, 
since actually all of the above exponent sets are needed 
to obtain a complete picture of 
capacity estimates.
Our purpose is to provide a unified approach 
which not only covers (and in many cases improves) all the previous 
capacity estimates in the literature, but also takes into
account the cases that have been overlooked in the past. 
We also indicate via 
Propositions~\ref{prop-sharp} and~\ref{prop-sharp-S} and 
numerous examples that our 
estimates are both natural and, in most cases, optimal.  
In addition, we hope that our work offers clarity and transparency to the proofs
of the previously known results. 

The following are some of our main results.
Here and 
later we often drop $x$ from the notation of the exponent sets 
when the point is fixed,
and moreover 
write e.g.\  $B_r=B(x,r)$.
For simplicity, we state the results here under the standard assumptions of doubling
and a Poincar\'e inequality, but in fact less is needed, as explained below.
Throughout the paper, we write $a \simle b$ if there is an implicit
 constant $C>0$ such that $a \le Cb$, where $C$ is independent of the 
essential parameters involved. We also write $a \simge b$ if $b \simle a$,
and $a \simeq b$ if $a \simle b \simle a$.
In particular, in Theorems~\ref{thm:intro}
and~\ref{thm:borderline-intro} below
the implicit constants are independent of $r$ and $R$, but
depend on $R_0$.

\begin{thm}\label{thm:intro}
Let\/ $0< R_0 < \frac{1}{4} \diam X$, $1\le p<\infty$, and assume that the measure $\mu$ is doubling
and supports a \p-Poincar\'e inequality.
\begin{enumerate}
\item  \label{it-intro-1}
If $p\in \interior\lQ_0$, then
\begin{equation} \label{eq-it-1-intro}
\cp(B_r,B_{R})\simeq \frac{\mu(B_r)}{r^{p}}
\quad \text{whenever\/ } 0<2r \le R \le R_0. 
\end{equation}
\item  \label{it-intro-2}
If $p\in \interior\uQ_0$, then
\begin{equation} \label{eq-it-2-intro}
\cp(B_r,B_{R})\simeq \frac{\mu(B_{R})}{R^{p}}
\quad \text{whenever\/ } 0<2r \le R \le R_0. 
\end{equation}
\end{enumerate}
\bigskip
%
Moreover, if \eqref{eq-it-1-intro} holds,
then $p \in \lQ_0$,
while if \eqref{eq-it-2-intro} holds,
then $p \in \uQ_0$. 
\end{thm}

Here and elsewhere  $\interior Q$ denotes the interior of a set $Q$.
Already unweighted $\R^n$ shows that 
$r$ needs to be bounded away from $R$ 
in order to have the upper bounds in~\eqref{eq-it-1-intro}
and~\eqref{eq-it-2-intro}
(hence $2r\le R$ above), and that the
lower estimate in~\eqref{eq-it-1-intro} (resp.\ \eqref{eq-it-2-intro}) does not hold in general
when $p \ge \sup\lQo$ (resp.\ $p \le \inf\uQo$), 
even if the borderline exponent is in the respective set. 
In these borderline cases $p=\max\lQo$ and $p=\min\uQo$ 
we instead obtain the following estimates
involving logarithmic factors.

\begin{thm}\label{thm:borderline-intro}
Let\/ $0< R_0 < \frac{1}{4} \diam X$, and 
assume that the measure $\mu$ is doubling
and supports a $\po$-Poincar\'e inequality for some\/ 
$1\leq \po <p$.

\begin{enumerate}
\item \label{it3-a-thm-intro}
If $p=\max\lQo$ and\/ $0<2r \le R \le R_0$, 
then
\begin{equation} \label{eq-it3-a-thm-intro}
\frac{\mu(B_{r})}{r^{p}}\biggl(\log\frac R r\biggr)^{1-p} \simle 
\cp(B_r,B_{R}) \simle \frac{\mu(B_{R})}{R^{p}}\biggl(\log\frac R r\biggr)^{1-p}.
\end{equation}

\item \label{it3-b-thm-intro}
 If $p=\min\uQo$ and\/ $0<2r \le R \le R_0$,
then
\begin{equation} \label{eq-it3-b-thm-intro}
\frac{\mu(B_{R})}{R^{p}}\biggl(\log\frac R r\biggr)^{1-p} \simle
\cp(B_r,B_{R}) \simle \frac{\mu(B_r)}{r^{p}}\biggl(\log\frac R r\biggr)^{1-p}.
\end{equation}
\end{enumerate}
\bigskip
%
Moreover,
if the lower bound in~\eqref{eq-it3-a-thm-intro} holds, then
$p\le\sup{\lQo}$, 
and if the lower bound in~\eqref{eq-it3-b-thm-intro} holds, then
$p\ge\inf{\uQo}$.
\end{thm}

See also \eqref{eq-upper-1} and \eqref{eq-upper-2} for
improvements of
the upper estimates of Theorem~\ref{thm:borderline-intro}.
Actually, 
Theorem~\ref{thm:borderline-intro}\,\ref{it3-a-thm-intro}
holds for all $p\in\lQo$ (resp.\ \ref{it3-b-thm-intro} for all $p\in\uQo$), 
but for $p$ in 
the interior of the respective exponent 
sets Theorem~\ref{thm:intro} gives better
estimates.
Let us also mention that 
for $p$ in between the $Q$-sets we obtain yet other
estimates depending on how close $p$ is to 
the corresponding $Q$-set, see Propositions~\ref{prop:simple upper bound}
and~\ref{prop:low bounds beyond borderline}.
Also these
estimates are sharp, as shown by 
Proposition~\ref{prop-sharp}.

We give related capacity estimates in terms of the
$S$-sets as well. 
In particular, we obtain the following 
almost characterization of when points have
zero capacity.
Here $\Cp(E)$ is the Sobolev capacity of $E\subset X$.

\begin{prop} \label{prop-zero-cap-intro}
Assume that $X$ is complete and 
that $\mu$ is doubling
and supports a \p-Poincar\'e inequality.
Let $B \ni x$ be a ball with $\Cp(X \setm B)>0$.

If\/ $1\le p\notin \uSo$ or\/ $1 < p \in \lSo$,
then $\Cp(\{x\})=\cp(\{x\},B)=0$.

Conversely, if 
$p\in \interior \uSo$, then 
$\Cp(\{x\})>0$ and\/ $\cp(\{x\},B)>0$.
\end{prop}

In the remaining borderline case, when $p = \inf \uSo \notin \lSo$, 
we show 
that the capacity can be either zero or 
nonzero, depending on the situation,
and thus the $S$-sets are not refined enough to give a complete characterization.

We also obtain similar results in terms of the $S_\infty$-sets,
which 
can be used to determine if the space $X$ is 
\p-parabolic or \p-hyperbolic; see Section~\ref{sect-S} for details.

For most of our estimates it is actually enough to require
that $\mu$ is both doubling and \emph{reverse-doubling} 
at the point $x$, and that a
Poincar\'e inequality holds for all balls centred at $x$.
Moreover, Poincar\'e inequalities and reverse-doubling
are only needed when
proving the lower bounds for capacities.
It is however worth pointing out that the examples
showing the sharpness of our estimates 
are based on \p-admissible weights on $\R^n$, and
so, even though our results hold in very general metric spaces,
it is essential to distinguish
the cases and define the exponent sets, as we do,
already in weighted $\R^n$. We construct our examples
with the help of a general machinery concerning radial
weights, explained 
in Section~\ref{sect:weights}.

Let us now give a brief account on some of the earlier results in the literature.
On unweighted $\R^n$, where $\lQo=\lSo=(0,n]$ and $\uQo=\uSo=[n,\infty)$, 
similar estimates (and precise calculations) are
well known, see e.g.\ Example~2.11 
in Heinonen--Kilpel\"ainen--Martio~\cite{HeKiMa},
which also contains an extensive treatise of potential theory on weighted $\R^n$,
including integral estimates for $A_p$-weighted capacities
with $p>1$
(Theorems~2.18 and~2.19 therein).
Theorem~3.5.6 in Turesson~\cite{Tur} provides essentially our
estimates for $p=1$ and $A_1$-weighted capacities in $\R^n$.
Estimates for general weighted Riesz capacities in $\R^n$ 
(including those equivalent to our capacities) were in somewhat 
different terms given in Adams~\cite[Theorem~6.1]{adams86}.

If the radii of the balls $B_r$ and $B_R$ are comparable, say $R=2r$, then
it is well known that the estimate $\cp(B_r,B_{2r})\simeq \mu(B_r)r^{-p}$ holds
(with implicit constants independent of $x$)
in metric spaces satisfying the doubling condition and a \p-Poincar\'e inequality,
see e.g.\ \cite[Lemma~2.14]{HeKiMa} for
weighted $\R^n$ and
J.~Bj\"orn~\cite[Lemma~3.3]{JBIll} or
Bj\"orn--Bj\"orn~\cite[Proposition~6.16]{BBbook}. 

Garofalo--Marola~\cite[Theorems~3.2 and~3.3]{GaMa} obtained 
essentially part~\ref{it-intro-1} of our 
Theorem~\ref{thm:intro} using an approach different 
from ours.
For the case $p = q(x):=\sup{\lQ(x)}$ they also gave 
estimates which are similar to
part~\ref{it3-a-thm-intro} of Theorem~\ref{thm:borderline-intro}.
However, they implicitly require that $q(x) \in {\lQ(x)}$ 
(i.e.\ $q(x)=\max\lQ(x)$) in their proofs,
and their estimates may actually fail if $q(x)\notin{\lQ}(x)$, as shown by 
Example~\ref{ex-log-weights}\,(c) below; the
same comment applies to their estimates in the case $p > q(x)$ as well.
There also seems
to be a slight problem in 
the proof of their lower bounds, 
since
the second displayed line at the beginning of the proof of Theorem~3.2 
in~\cite{GaMa} does not in general follow from the first line, as can be seen by 
considering e.g.\ $u(x)=\max\{0,\min\{1,1+j(r-|x|)\}\}$ in $\R^n$ and letting $j\to\infty$.
Instead, this estimate can be derived directly from a 1-Poincar\'e inequality (see 
M\"ak\"al\"ainen~\cite{Makalainen}), which is a stronger assumption than the
\p-Poincar\'e inequality assumed in~\cite{GaMa} (and in the present work).

Also Adamowicz--Shanmugalingam~\cite{adsh} have given related estimates
in metric spaces. They state their results in terms of the \p-modulus of curve families,
but it is known
that the \p-modulus coincides with the variational \p-capacity,
provided that $X$ is complete and $\mu$
is doubling and supports a \p-Poincar\'e inequality, see e.g.\ 
Heinonen--Koskela~\cite{HeKo98}, Kallunki--Shanmugalingam~\cite{KaSh} and
Adamowicz--Bj\"orn--Bj\"orn--Shan\-mu\-ga\-lin\-gam~\cite{ABBSprime}.
In the setting considered in~\cite{adsh} this equivalence is not known in general.
While it is always true that the \p-modulus is majorized by 
the variational \p-capacity, the converse is only
known under the assumption of a \p-Poincar\'e inequality, which is not
required for the upper bounds in~\cite{adsh} nor here.
At the same time, the test functions in~\cite{adsh} are admissible also for $\cp$,
showing that their estimates apply also to the variational \p-capacity.
For $p\in\interior\lQo$, 
Theorem~3.1 in~\cite{adsh} provides an upper bound that can be
seen to be weaker than~\eqref{eq-it-1-intro}. 
In the borderline case $p=\max\lQ_0$ (when it is attained), the upper estimate~(3.6) 
in~\cite{adsh} coincides with our~\eqref{eq-it3-a}.
Under the assumption that the space $X$ is Ahlfors $Q$-regular
and supports a \p-Poincar\'e inequality, they also prove lower bounds for capacities. 
For $p>Q$, the lower bound in~\cite[Theorem~4.3]{adsh}
coincides with the one in 
Theorem~\ref{thm:intro}\,\ref{it-intro-2}, but for $p\le Q$
the lower bound in~\cite[Theorem~4.9]{adsh}
is weaker than our estimates~\eqref{eq-it-1-intro} and~\eqref{eq-it3-a-thm-intro}.

Neither \cite{adsh} nor \cite{GaMa} contain any results
similar to ours for 
$p \in \uQo$, 
or in terms of $q \in \uQo$ for $p\notin\uQo$,
or involving the $S$-sets.

As mentioned above, 
\p-capacity and \p-modulus estimates are closely related,
and our estimates trivially give estimates for the \p-modulus in all
cases when they coincide, 
e.g.\ when 
$X$ is complete and $\mu$
is doubling and supports a \p-Poincar\'e inequality,
see above. 
Moreover, our upper estimates are trivially upper 
bounds of the \p-modulus in all cases.
We do not know if our lower estimates of the capacity are also
lower bounds for the \p-modulus, but neither
do we know of any example when the \p-modulus is 
strictly smaller than the \p-capacity.

Let us also mention that
earlier capacity estimates in Carnot groups and Carnot--Carath\'eodory spaces
can be found in Heinonen--Holopainen~\cite{HeHol} and 
in Capogna--Danielli--Garofalo~\cite{CaDaGa2}, respectively.
In~\cite{CaDaGa2}, the estimates are then applied to yield information on the
behaviour of singular solutions of certain quasilinear equations near the
singularity; see also Danielli--Garofalo--Marola~\cite{DaGaMa} for related
results in more general settings.  
In addition, Holopainen--Koskela~\cite{HoKo} provided a lower bound 
for the variational capacity in terms of the volume growth in Riemannian manifolds, 
as well as some related estimates in general metric spaces,
which in turn are related to the parabolicity and hyperbolicity of the
space.
Capacities defined by nonlinear potentials on homogeneous groups were
considered by Vodop{\cprime}yanov~\cite{Vodop} and some estimates in terms of
$A_p$-weights  were given in Proposition~2 therein.

The outline of the paper is as follows:
In Section~\ref{sect-sets-of-exponents}
we introduce some basic terminology and discuss
the exponent sets under consideration in this paper,
while in Section~\ref{sect-ex} we give some key examples
demonstrating various possibilities for the exponent sets.
These examples will later, in Section~\ref{sect-sharpness},
be used to show sharpness of our estimates.

In Section~\ref{sect-metric} we
introduce the necessary background for metric space analysis,
such as capacities and Newtonian (Sobolev) spaces based
on upper gradients.
Towards the end of the section we obtain a few new results and also 
the basic estimate used to obtain all our lower capacity bounds
(Lemma~\ref{lem:chain estimate}).

Sections~\ref{sect-5}--\ref{sect-S} are all 
devoted to the various capacity estimates.
In Section~\ref{sect-5} we obtain upper bounds, 
which are easier to obtain than lower bounds and in particular
require less assumptions on the space.
Lower bounds related to the $Q$-sets are established 
in Sections~\ref{sect:interior} and~\ref{sect:borderline},
the latter containing some more involved borderline cases,
while in Section~\ref{sect-S} we
study (upper and lower) estimates in terms of the $S$-sets
and in particular prove 
Proposition~\ref{prop-zero-cap-intro}
and the parabolicity/hyperbolicity results mentioned above.

The sharpness of most of our estimates
(but for some borderline cases)
is demonstrated in Section~\ref{sect-sharpness}.
Here we extend our discussion of the examples already introduced
in Section~\ref{sect-ex} using the 
capacity formula for radial weights on $\R^n$ given in
Proposition~\ref{prop-cap-formula-radial}.
This formula enables us to 
compute explicitly
the capacities in the examples, and thus we can
make comparisons with the bounds given by the more 
general estimates from Sections~\ref{sect-5}--\ref{sect-S}. 
We also obtain  stronger and more theoretical
sharpness results in Propositions~\ref{prop-sharp}
and~\ref{prop-sharp-S}.

The final Section~\ref{sect:weights} is devoted to proving 
the capacity formula mentioned above, and along the way we obtain
some new results on quasiconformality of radial stretchings
and on \p-admissibility of radial weights.

\begin{ack}
A.B.\ and J.B.\
were supported by the Swedish Research Council.
J.L.\ was supported by the Academy of Finland (grant no.\ 252108) and the 
V\"ais\"al\"a Foundation of the Finnish Academy of Science and Letters.
Part of this research was done during several visits
of J.L.\ 
to Link\"opings universitet in 2012--13,
and while A.B.\ and J.B.\ visited Institut Mittag-Leffler in 2013.
We wish to thank these institutions for their 
kind hospitality.
\end{ack}

\section{Exponent sets} 
\label{sect-sets-of-exponents}

We assume throughout the paper that $1 \le p<\infty$ 
and that $X=(X,d,\mu)$ is a metric space equipped
with a metric $d$ and a positive complete  Borel  measure $\mu$ 
such that $0<\mu(B)<\infty$ for all balls $B \subset X$. We 
adopt the convention that balls are nonempty and open.
The $\sigma$-algebra on which $\mu$ is defined
is obtained by the completion of the Borel $\sigma$-algebra.
It follows that $X$ is separable.

\begin{deff}
We say that the measure
$\mu$ is \emph{doubling at $x$}, if there 
is a constant $C>0$ such that whenever $r>0$, we have
\begin{equation}\label{eq:doubling-x0}
  \mu(B(x,2r))\le C \mu(B(x,r)).
\end{equation}
Here $B(x,r)=\{y \in X : d(x,y)<r\}$.
If \eqref{eq:doubling-x0} holds with the same constant $C>0$
for all $x \in X$, we say that $\mu$ is 
\emph{\textup{(}globally\/\textup{)} doubling}.
\end{deff}

The global doubling condition is often assumed in the metric space literature, but for
our estimates it will be enough to assume that $\mu$ is doubling at $x$.
Indeed, this will be a standing assumption for us
from Section~\ref{sect-5} onward.

\begin{deff}
We say that the measure $\mu$ is \emph{reverse-doubling at $x$}, if there
are constants $\ga,\tau>1$ such that 
\begin{equation}\label{eq:rev-doubling-x0}
  \mu(B(x,\tau r))\ge \ga \mu(B(x,r))
\end{equation}
holds for all $0<r\le \diam X/2\tau$.
\end{deff}

If $X$ is connected (or uniformly perfect) 
and $\mu$ is globally doubling, then
$\mu$ is also reverse-doubling at every point, with uniform constants;
see e.g.\ Corollary~3.8 in~\cite{BBbook}.
If $\mu$ is merely doubling at $x$, then
the reverse-doubling at $x$ does not follow automatically
and has to be imposed separately whenever needed.

If both~\eqref{eq:doubling-x0} and~\eqref{eq:rev-doubling-x0}
hold, then an iteration of these conditions
shows that there exist $q, q'>0$ and $C,C'>0$ 
such that 
\begin{equation}\label{eq:measure lower-upper}
C'\Bigl(\frac rR\Bigr)^{q'} \le \frac{\mu(B(x,r))}{\mu(B(x,R))}
\le C\Bigl(\frac rR\Bigr)^{q}
\end{equation}
whenever $0<r\le R<{2}\diam X$.
More precisely, the doubling inequality~\eqref{eq:doubling-x0}
leads to the first inequality, 
while the 
reverse-doubling \eqref{eq:rev-doubling-x0} yields the 
second inequality 
of~\eqref{eq:measure lower-upper}.
Recall also that the measure $\mu$ (and also the space $X$) is said to be 
\emph{Ahlfors $Q$-regular} if $\mu(B(x,r))\simeq r^Q$
for every $x\in X$ and all $0<r<2\diam X$.
This in particular implies that~\eqref{eq:measure lower-upper} holds with $q=q'=Q$.

The inequalities in \eqref{eq:measure lower-upper} will be of fundamental importance
to us.
Note that in \eqref{eq:measure lower-upper} one necessarily has $q'\ge q$ 
and that
there can be a gap between the exponents, as demonstrated by Example~\ref{ex1} below.
Garofalo--Marola~\cite{GaMa} introduced  the \emph{pointwise dimension}
$q(x)$ (called $Q(x)$ therein)
as the supremum of all $q>0$ such that
the second inequality in~\eqref{eq:measure lower-upper} holds
for some $C_q>0$ and all $0<r\le R<\diam X$. Furthermore,
Adamowicz--Bj\"orn--Bj\"orn--Shan\-mu\-ga\-lin\-gam~\cite{ABBSprime}
defined the \emph{pointwise dimension set} $Q(x)$
consisting  of all  $q>0$
for which there are constants $C_q>0$ and $R_q>0$ such that 
the second inequality in~\eqref{eq:measure lower-upper} holds
for all $0<r\le R\le R_q$. 
It was shown in~\cite[Example~2.3]{ABBSprime} that it is possible
to have $Q(x)=(0,q)$ for some $q$, that is,
the end point $q$ need  not be contained in the interval $Q(x)$.
Alternatively see Example~\ref{ex-log-weight-Q}
below.

For us it will be important to make even further distinctions.
We consider the exponent sets $\lQo$, $\lSo$, $\uSo$ and $\uQo$
from the introduction.
The pointwise dimension of Garofalo--Marola~\cite{GaMa} is then
$q(x) = \sup \lQ(x)$, where $\lQ(x)$ is a global version of $\lQo(x)$
(see below for the precise definition),
and the pointwise dimension set of~\cite{ABBSprime} is $Q(x)=\lQo(x)$
(to see this, one should also appeal to Lemma~\ref{lem-2}).
Recall that we often drop $x$ from the  notation,
and write $B_r=B(x,r)$.

If $\mu$ is doubling at $x$ (resp.\ reverse-doubling at $x$), then 
$\uQo\neq\emptyset$ 
(resp.\ $\lQo\neq\emptyset$), by \eqref{eq:measure lower-upper}.
The sets $\lQo$ and $\lSo$ are then intervals of the form $(0,q)$ or $(0,q]$,
whereas $\uQo$ and $\uSo$ are intervals of the form $(q,\infty)$ or
$[q,\infty)$.
Whether the end point is or is not included in the 
respective intervals will be important in many situations.

We start our discussion of the exponent sets by three lemmas concerning their
elementary properties. 
Note that Lemmas~\ref{lem-1}--\ref{lem-2} and~\ref{lem:infty sets invariant}
hold for arbitrary measures, without
assuming any type of doubling.

\begin{lem} \label{lem-1}
It is true that 
\[
    \lQo \subset \lSo 
    \quad \text{and} \quad
    \uQo \subset \uSo.
\]
Moreover, $\lSo \cap \uSo$ contains at most one point, and when
it is nonempty, $\lQo=\lSo$ and $\uQo=\uSo$.
\end{lem}

\begin{proof}
If $q \in \lQo$, then
\(
    \mu(B_r) \le C_q \mu(B_1) r^q, 
\)
and thus $q \in \lSo$. 
Similarly $\uQo \subset \uSo$.

For the second part, let $q \in \lSo \cap \uSo$.
Then $\mu(B_r) \simeq r^q$ and it follows
that $q \in \lQo$ and $q \in \uQo$.
That $\lQo=\lSo$ and $\uQo=\uSo$ thus follows
from the first part.
\end{proof}

The following two lemmas show that the bound $1$ on the radii in the definitions
of the exponent sets can equivalently be replaced by any other fixed bound $R_0$.
They also provide formulas for the borderline exponents in the $S$-sets and 
estimates for the borderline exponents in the $Q$-sets.
Examples~\ref{ex-Cantor-staircase} and~\ref{ex-q-not-end-points} show that finding
the exact end points of the $Q$-sets may be rather subtle.

\begin{lem} \label{lem-3}
Let $q,R_0>0$.
Then $q \in \lSo$ if and only if there 
is a constant $C>0$ such that
\begin{equation} \label{eq-muBr-R_0}
        \mu(B_r)  \le C r^q 
       \quad  \text{for } 0 < r \le R_0.
\end{equation}
Similarly, $q \in \uSo$ if and only if there 
is a constant $C>0$ such that
\[
        \mu(B_r)  \ge C r^q 
       \quad  \text{for } 0 < r \le R_0.
\]
Furthermore,
let
\[
          q_0=\liminf_{r \to 0} \frac{\log \mu(B_r)}{\log r}
          \quad \text{and} \quad
          q_1=\limsup_{r \to 0} \frac{\log \mu(B_r)}{\log r}.
\]
Then $\lSo=(0,q_0)$ or $\lSo=(0,q_0]$,
and  $\uSo=(q_1,\infty)$ or $\uSo=[q_1,\infty)$.
\end{lem}

\begin{proof}
For the first part, assume that $q \in \lSo$. 
We may assume that $R_0>1$.
If $1 \le r < R_0$, then 
\[
           \mu(B_r) \le \mu(B_{R_0}) \le \mu(B_{R_0}) r^q,
\]
i.e.~\eqref{eq-muBr-R_0} holds with $C:=\max\{C_q,\mu(B_{R_0})\}$.
The converse implication is proved similarly.

For the last part, after taking logarithms 
we see that $q \in \lSo$ if and only if there is $C_q$ such 
that
\[
            q \le \frac{\log \mu(B_r)}{\log r} - \frac{\log C_q}{\log r}
            \quad \text{for } 0 <r < 1,
\]
which is easily seen to be possible if $q < q_0$, and impossible
if $q> q_0$.
The proofs for $\uSo$ are similar.
\end{proof}

\begin{lem} \label{lem-2}
Let $q,R_0>0$.
Then $q \in \lQo$ if and only if there 
is a constant $C>0$ such that
\begin{equation}   \label{eq-Q-R_0}
        \frac{\mu(B_r)}{\mu(B_R)}  \le C \Bigl(\frac{r}{R}\Bigr)^q 
        \text{ for } 0 < r < R \le R_0.
\end{equation}
The corresponding statement for $\uQo$ is also true.

Assume furthermore that $f(r):=\mu(B_r)$ is locally absolutely continuous 
on\/ $(0,\infty)$ and let
\[ 
\loq = \essliminf_{r \to 0} \frac{rf'(r)}{f(r)}
\quad \text{and} \quad
\uq=\esslimsup_{r \to 0} \frac{rf'(r)}{f(r)}.
\] 
Then
\[
    (0,\loq) \subset \lQo \subset (0,\uq]
\quad \text{and} \quad
    (\uq,\infty) \subset \uQo \subset [\loq,\infty).
\]
\end{lem}

The following example shows that the assumption that $f$ is locally absolutely
continuous in Lemma~\ref{lem-2} is not redundant.

\begin{example}  \label{ex-Cantor-staircase} 
Let $X$ be the usual Cantor ternary set, defined as a subset of $[0,1]$ and
equipped with the normalized $d$-dimensional Hausdorff measure $\mu$ with 
$d=\log 2 /{\log 3}$. 
Let $x=0$.
Then $f(r)=\mu(B_r)$ will be the Cantor staircase function
which is not absolutely continuous.
(See Dovgoshey--Martio--Ryazanov--Vuorinen~\cite{CantorFn} 
for the history of the Cantor staircase function.)
At the same time, $\mu$ is Ahlfors $d$-regular and hence $\lSo=\lQo=(0,d]$
and $\uSo=\uQo=[d,\infty)$, while $\loq=\uq=0$.

On the other hand if $X=\R^n$ is equipped with a weight $w$ and 
$d\mu=w\, dx$, then
$f$ automatically is locally absolutely continuous.
In particular, this is true if $w$ is a \p-admissible weight.
We do not know if $f$ is always locally absolutely continuous whenever
$\mu$ is both globally doubling and 
supports a global Poincar\'e inequality.
\end{example}

\begin{proof}[Proof of Lemma~\ref{lem-2}]
We prove that $q\in\lQo$ implies~\eqref{eq-Q-R_0}.
The proofs of the converse implication and for $\uQo$ are similar.
We may assume that $R_0>1$.
If $1\le r< R\le R_0$, then
\[
\frac{\mu(B_r)}{\mu(B_R)} \le 1 = R_0^q \biggl( \frac{1}{R_0} \biggr)^q  
\le R_0^q \Bigl(\frac{r}{R}\Bigr)^q.
\]
For $r\le 1\le R\le R_0$ we instead have
\[
\frac{\mu(B_r)}{\mu(B_R)} \le \frac{\mu(B_r)}{\mu(B_1)} \le C_q r^q 
\le C_q R_0^q \Bigl(\frac{r}{R}\Bigr)^q.
\]
Thus,~\eqref{eq-Q-R_0} holds whenever $R\ge1$. 
For $R\le1$ the claim follows directly from
the assumption $q\in\lQo$.

Next assume that $f$ is locally absolutely continuous and let $q \in(0,\loq)$.
Then $h(r)=\log f(r)$
is also locally absolutely continuous and $h'(r)=f'(r)/f(r)$.
By assumption there is $\Rt$ such that $\rho h'(\rho)>q$ for a.e.\ $0<\rho \le \Rt$.
Since $h$ is locally absolutely continuous, we have
for $0<r<R \le \Rt$ that
\[
    \log \frac{f(R)}{f(r)} 
     = h(R)- h(r) = \int_r^R h'(\rho) \,d\rho
    \ge \int_r^R \frac{q}{\rho}\,d \rho
    = \log  \biggl(\frac{R}{r}\biggr)^q,
\]
and thus
\[
     \frac{\mu(B_r)}{\mu(B_R)} \le \Bigl(\frac{r}{R}\Bigr)^q.
\]
By the first part, with $R_0=\Rt$, we get that 
 $q \in \lQ_0$.
Hence $(0,\loq) \subset \lQo$.
The proof that    $ (\uq,\infty) \subset \uQo$ is analogous.
The remaining inclusions follow from these inclusions  together with
the fact that $\lQo \cap \uQo$ contains at most one point (by
Lemma~\ref{lem-1}).
\end{proof}

The following example shows that $\loq$ and $\uq$ (from Lemma~\ref{lem-2}) 
need not be the end points of $\lQ_0$ and $\uQ_0$.

\begin{example}  \label{ex-q-not-end-points}
Let $f$ be given for $r\in(0,\infty)$ by
\[
f(r)=\begin{cases}
    a_k r^{n-1}, & \text{if } 4^{-k} \le r \le 2\cdot 4^{-k}, 
                     \ k\in\Z,\\
    \displaystyle \frac {r^{n+1}}{a_k}, 
                  & \text{if } 2\cdot 4^{-k} \le r \le 4\cdot 4^{-k}, 
                     \ k\in\Z,\\
        \end{cases}
\] 
where $a_k = 2\cdot 4^{-k}$ and $n\ge1$.
Note that $f$ is increasing and locally Lipschitz.
For a.e.\ $x\in\R^n$ set
\[
w(x) =  \frac{f'(|x|)}{\omega_{n-1} |x|^{n-1}},
\]
where $\omega_{n-1}$ is the surface area of the $(n-1)$-dimensional sphere in $\R^n$.
With this choice of $w$ we have 
\[
f(r)=\omega_{n-1} \int_0^r w(\rho)\rho^{n-1}\,d\rho = \mu(B_r),
\]
where $d\mu = w\,dx$.
Since
\[
f'(r)=\begin{cases}
    (n-1)a_k r^{n-2}, & \text{if } 4^{-k}<r<2\cdot 4^{-k}, 
                     \ k\in\Z,\\
    \displaystyle \frac {n+1}{a_k} r^{n}, 
              & \text{if } 2\cdot 4^{-k}<r<4\cdot 4^{-k}, \ k\in\Z,\\
        \end{cases}
\] 
and $r\simeq a_k$ on $(4^{-k}, 4\cdot 4^{-k})$, we see that $w\simeq 1$
on $\R^n$, i.e.\ $\mu$ is comparable to the Lebesgue measure.
In particular, $\mu$ is Ahlfors $n$-regular and supports a 
global $1$-Poincar\'e 
inequality, $\lQ_0=(0,n]$ and $\uQ_0=[n,\infty)$. 

At the same time, considering $r \in (4^{-k}, 2\cdot 4^{-k})$ and
$r \in (2\cdot 4^{-k}, 4\cdot 4^{-k})$, respectively, gives
\[
{\essliminf_{r\to0}} \frac{r f'(r)}{f(r)} = n-1
\quad \text{and} \quad
{\esslimsup_{r\to0}} \frac{r f'(r)}{f(r)} = n+1.
\]
It is easy to construct a similar example with a continuous weight $w$.
\end{example}

If $X$ is unbounded, we will consider the following 
exponent sets at $\infty$  
for results in large balls and with respect to the whole space:
\begin{align*}
  \lQi(x)
        & :=\Bigr\{q>0 : \text{there is $C_q$ so that } 
        \frac{\mu(B(x,r))}{\mu(B(x,R))}  \le C_q \Bigl(\frac{r}{R}\Bigr)^q 
        \text{ for } 1 \le r < R
        \Bigl\}, \\
  \lSi(x)
          & :=\{q>0 : \text{there is $C_q$ so that } 
        \mu(B(x,r))  \ge C_q r^q 
        \text{ for } r\ge1 
        \}, \\
  \uSi(x)
          & :=\{q>0 : \text{there is $C_q$ so that } 
        \mu(B(x,r))  \le C_q r^q 
        \text{ for }  r\ge1 
        \}, \\
  \uQi(x)
        & :=\Bigr\{q>0 : \text{there is $C_q$ so that } 
        \frac{\mu(B(x,r))}{\mu(B(x,R))}  \ge C_q \Bigl(\frac{r}{R}\Bigr)^q 
        \text{ for } 1 \le r < R  
        \Bigl\}.
\end{align*}
Note that the inequality in $\lSi(x)$ is reversed from the one in $\lSo(x)$,
and similarly for $\uSi(x)$.
This guarantees that $\lSi=(0,q)$ or $\lSi=(0,q]$ and $\uSi=(q,\infty)$ or 
$\uSi=[q,\infty)$, rather than the other way round,
and also that $\lQi \subset \lSi$ and $\uQi \subset \uSi$.

Lemmas~\ref{lem-1}--\ref{lem-2} above have direct counterparts for these 
exponent sets at~$\infty$.
In addition, Lemma~\ref{lem:infty sets invariant} below 
shows that these sets are actually
independent of the
point $x\in X$, and thus the sets $\lQi$, $\lSi$, $\uSi$ and $\uQi$ are
well defined objects for the whole space $X$, not merely a
short-hand notation (with a fixed base point $x\in X$)
as in the case of $\lQo$, $\lSo$, $\uSo$ and $\uQo$.
Note, however, that in general for instance the set $\lSi$ is
different from the set
\[
\{q>0 : \text{there is $C_q$ so that } 
        \mu(B(x,r))  \le C_q r^q 
        \text{ for every } x\in X \text{ and all } r \ge 1 
        \},
\]
since the constant $C_q$ in the definition of $\lSi$ 
is allowed to depend on the point $x$.
This can be seen e.g.\ by letting 
$w(x)=\log(2+|x|)$, 
which is  a $1$-admissible weight on $\R^n$ by 
Proposition~\ref{prop-suff-admiss-w} below.
Recall that a weight $w$ in $\R^n$
is \p-admissible, $p\ge1$, if the measure $d\mu={w}\,dx$
is globally doubling and supports
a global \p-Poincar\'e inequality.

\begin{lem}\label{lem:infty sets invariant}
 Let $X$ be unbounded and fix $x\in X$.
 Then, for every $y\in X$, we have
 $\lQi(x)=\lQi(y)$, $\lSi(x)=\lSi(y)$, $\uSi(x)=\uSi(y)$ and $\uQi(x)=\uQi(y)$.
\end{lem}

\begin{proof}
Let $y\in X$. 
By (the $\infty$-versions of) Lemmas~\ref{lem-3} and~\ref{lem-2} it is enough to
verify the definitions of the exponent sets for $R>r\ge 2d(x,y)$.
In this case we have
$B(x,r/2)\subset B(y,r)\subset B(x,2r)$ and similarly for $B(y,R)$.
Hence 
\[
\frac{\mu(B(x,r/2))}{\mu(B(x,2R))} \le \frac{\mu(B(y,r))}{\mu(B(y,R))} 
\le \frac{\mu(B(x,2r))}{\mu(B(x,R/2))},
\]
which shows that the inequalities in the definitions of the exponent sets at $\infty$
hold for $y$ if and only if they hold for $x$.
\end{proof}

Finally, when we want to be able to treat both large and small balls uniformly
we need to use the sets
\begin{align*}
   \lQ(x):=\lQo(x) \cap \lQi 
   \quad 
   \text{and} 
   \quad
   \uQ(x):=\uQo(x) \cap \uQi.
\end{align*}
If $X$ is bounded, we simply set  $\lQ:=\lQo$
and $\uQ:=\uQo$.

\begin{remark}
Let $k(t)=\log \mu(B_{e^t})$. Then it is easy to show that
$q \in \lQo$ and $q' \in \uQo$ if and only if there is a constant $C$
such that
\[
          q (T-t) - C \le      k(T) - k(t) \le q'(T-t) + C,
          \quad \text{if } t < T < 0,
\]
or in other terms
\[
          q |T-t| - C \le      |k(T) - k(t)| \le q'|T-t| + C,
          \quad \text{if } t, T < 0,
\]
i.e.\ $k$ is a $(q,q',C)$-rough quasiisometry on $(-\infty,0)$ for some $C$.
Similarly, if $X$ is unbounded, 
then $k$ is 
a $(q,q',C)$-rough quasiisometry on $(0,\infty)$ 
(resp.\ on $\R$) for some $C$
if and only if $q \in \lQi$ and $q' \in \uQi$
(resp.\ $q \in \lQ$ and $q' \in \uQ$).
Much of the current literature on rough 
quasiisometries call such maps quasiisometries,
but we have chosen to follow the terminology of
Bonk--Heinonen--Koskela~\cite{BHK}  to avoid confusion with biLipschitz maps.
\end{remark}

\section{Examples of exponent sets}
\label{sect-ex}

In this section we give various examples of the exponent sets.
In particular, we shall see that the end points of the four exponent sets can 
all be different (Examples~\ref{ex1} and~\ref{ex-abcd-alt}) 
and that the borderline
exponents may or may not belong to the sets (Examples~\ref{ex-log-weight-Q} 
and~\ref{ex-S-touch}). 

Our examples are based on radial weights in $\R^n$, and all the weights 
we consider
are in fact 1-admissible, i.e.\ they are globally doubling
and support a global 1-Poincar\'e inequality on $\R^n$.
Later in Section~\ref{sect-sharpness} these 
weights 
will be used to demonstrate 
the sharpness of several of our capacity estimates.
In Section~\ref{sect:weights} we give
a general sufficient condition for 1-admissibility of radial weights. 

For simplicity, we write e.g.\ $\log^\be r:=(\log r)^\be$.

\begin{example}   \label{ex-log-weight-Q}
Consider $\R^n$, $n\ge2$, equipped with the measure $d\mu=w(|y|)\,dy$, where
\[
      w(\rho)=\begin{cases}
          \rho^{p-n} \log^\be (1/\rho), & \text{if } 0<\rho\le1/e, \\
          \rho^{p-n} , & \text{otherwise}.
        \end{cases}
\]
Here $p\ge1$ and $\be\in\R$ is arbitrary.
Fix $x=0$ and write $B_r=B(0,r)$.
Then it is easily verified that for $r\le1/e$ we have 
$\mu(B_r)\simeq r^p\log^\be (1/r)$.
Letting $r\to0$ in the definition of the 
exponent sets
shows that 
\[
    \lSo=\lQo=\lQ = \begin{cases}
      (0,p], & \text{if } \be \le 0, \\
      (0,p), & \text{if } \be > 0, 
    \end{cases}
\quad \text{and} \quad
    \uSo=\uQo=\uQ = \begin{cases}
      (p,\infty), & \text{if } \be <0, \\
      [p,\infty), & \text{if } \be \ge 0. 
    \end{cases}
\]
In both cases $\sup\lQ=\inf\uQ=p$, 
but only one of these is 
attained (when $\be\ne0$).
Letting instead 
\[  
    w(\rho)=\begin{cases}
   \rho^{p-n} \log^\be \rho & \text{for } \rho\ge e, \\
   \rho^{p-n}, & \text{otherwise},
   \end{cases}
\]
gives again $\sup\lQ=\inf\uQ=p$, but if $\be>0$ it is now $\sup\lQ$ that is attained, 
while for $\be<0$ only $\inf\uQ$ is attained.
\end{example}

\begin{example} \label{ex1}
We are now going to create an example of a $1$-admissible weight 
in $\R^2$ with
\begin{equation} \label{eq-ex1}
\lQ=\lQo=(0,2], \quad
\lSo=(0,3], \quad
\uSo=\bigl[\tfrac{10}{3},\infty\bigr)
    \quad \text{and} \quad
\uQ=\uQo=[4,\infty),
\end{equation}
showing that the four end points can all be different.

Let $\alp_k=2^{-2^k}$ and
 $\be_k= \alp_k^{3/2}=2^{-3\cdot 2^{k-1}}$, $k=0,1,2,\ldots$\,.
Note that $\alp_{k+1}= \alp_k^2$. 
In $\R^2$ we fix $x=0$ and consider the measure $d\mu=w(|y|)\,dy$, where
\[
      w(\rho)=\begin{cases}
                     \alp_{k+1}, & \text{if } \alp_{k+1} \le \rho \le \be_k, 
                     \ k=0,1,2,\ldots,\\
                     \rho^2/ \alp_{k}, & \text{if } \be_k \le \rho \le \alp_{k}, 
                     \ k=0,1,2,\ldots,\\
                     \rho, & \text{if } \rho \ge \tfrac{1}{2}.
        \end{cases}
\] 
Then
\[
          \frac{\rho w'(\rho)}{w(\rho)}
          = \begin{cases}
                     0 , & \text{if } \alp_{k+1} < \rho < \be_k, 
                     \ k=0,1,2,\ldots,\\
                     2, & \text{if } \be_k < \rho < \alp_{k}, 
                     \ k=0,1,2,\ldots,\\
                     1, & \text{if } \rho > \tfrac{1}{2},
        \end{cases}
\]
and thus $w$ is $1$-admissible by Proposition~\ref{prop-suff-admiss-w}.
We next have that 
\begin{equation} \label{eq-ex1-2}
         \mu(B_{r} \setm B_{\alp_{k+1}}) 
          \simeq \int_{\alp_{k+1}}^{r}  w(\rho)\rho \, d\rho 
          = \frac{\alp_{k+1}}{2} (r^2-\alp_{k+1}^2),
          \quad \text{if } \alp_{k+1} \le r \le \be_k.
\end{equation}
In particular,
\[
         \mu(B_{\be_k} \setm B_{\alp_{k+1}}) 
          \simeq 
           \frac{\alp_{k+1}}{2} (\beta_k^2-\alp_{k+1}^2)
          = \frac{\alp_{k}^5(1-\alp_k)}{2} 
          \simeq \alp_{k}^5.
\]
For $\be_k \le r \le \alp_k$ we instead have 
\begin{equation} \label{eq-ex1-2b}
         \mu(B_{r} \setm B_{\be_{k}}) 
          \simeq \int_{\be_k}^{r}  w(\rho)\rho \, d\rho 
          = \frac{r^4-\be_k^4}{4\alp_k},
\end{equation}
and thus
\[
         \mu(B_{\alp_k} \setm B_{\be_{k}}) 
          \simeq  \frac{\alp_{k}^4-\be_k^4}{4\alp_k} 
          \simeq \alp_{k}^3.
\]
It follows that
\begin{equation} \label{eq-ex1-3b}
    \mu(B_{\be_k}) \simeq \alp_k^5 +\alp_k^6 + \alp_k^{10} + \alp_k^{12} + \ldots 
      \simeq   \alp_k^5=\be_k^{10/3}
\end{equation}
and 
\begin{equation} \label{eq-ex1-3}
    \mu(B_{\alp_k}) \simeq \alp_k^3 + \alp_k^5 \simeq \alp_k^3.
\end{equation}
Since $w(\rho) \le \rho$ for all $\rho$, we have that $\mu(B_r) \simle r^3$ for
all $r$, which together with \eqref{eq-ex1-3} shows
that $\lSo=(0,3]$.

From the estimates \eqref{eq-ex1-3} and 
\eqref{eq-ex1-2} we obtain
\begin{equation} \label{eq-ex1-4}
      \mu(B_r) \simeq \alp_{k+1} r^2, \quad
      \text{if } \alp_{k+1}  \le r \le \be_k.
\end{equation}
Indeed, when $\alp_{k+1} \le r \le 2 \alp_{k+1}$ this
follows directly from \eqref{eq-ex1-3}, and for 
$2 \alp_{k+1} \le r \le \be_k$ we use 
\eqref{eq-ex1-2} to 
get a lower bound, while the upper bound follows from \eqref{eq-ex1-2}
together with \eqref{eq-ex1-3}.
In particular, we get that
\begin{equation} \label{eq-ex1-4b}
      \mu(B_r) \simeq \alp_{k+1} r^2 
      = \be_k^{4/3} r^2 \ge r^{10/3}, \quad
      \text{if } \alp_{k+1}  \le r \le \be_k.
\end{equation}
Estimating similarly, using instead \eqref{eq-ex1-2b}
and \eqref{eq-ex1-3b}, shows that
\begin{equation}    \label{eq-Br-le-al}
      \mu(B_r) \simeq \frac{r^4}{\alp_k}  
      = \frac{r^4}{\be_k^{2/3}}  \ge r^{10/3}, \quad
      \text{if } \be_{k}  \le r \le \alp_k.
\end{equation}
We conclude from the 
last two estimates 
and from \eqref{eq-ex1-3b} 
that $\uSo=\bigl[\tfrac{10}{3},\infty\bigr)$.

Next, we see from \eqref{eq-ex1-4} and \eqref{eq-Br-le-al} 
that
\begin{equation} \label{eq-ex1-7}
          \frac{\mu(B_r)}{\mu(B_R)} \simeq 
          \begin{cases}
\displaystyle \Bigl(\frac{r}{R}\Bigr)^2,
     &       \text{if } \alp_{k+1}  \le r \le R \le \be_k, \\[2mm]
\displaystyle \Bigl(\frac{r}{R}\Bigr)^4,
     &       \text{if } \be_{k}  \le r \le R \le \alp_k.
     \end{cases}
\end{equation}
Hence, if
$\alp_{k+1} \le r \le \be_k \le R \le \alp_k$, then
\[
  \frac{\mu(B_r)}{\mu(B_R)} = \frac{\mu(B_r)}{\mu(B_{\be_k})}\frac{\mu(B_{\be_k})}{\mu(B_R)}
  \simeq \biggl(\frac{r}{\be_k}\biggr)^2\biggl(\frac{\be_k}{R}\biggr)^4
  = \frac{r^2 \be_k^2}{R^4} 
\]
and thus
\[
  \Bigl(\frac{r}{R}\Bigr)^4 \simle \frac{\mu(B_r)}{\mu(B_R)} 
   \simle \Bigl(\frac{r}{R}\Bigr)^2.
\]
It follows from \eqref{eq-ex1-7} that this estimate holds also 
in the remaining cases when $\alp_{k+1} \le r  \le R \le \alp_k$.
Finally, if $\alp_{j+1} \le r \le \alp_j \le \alp_{k+1} \le R \le \alp_k$,
then
\[
  \frac{\mu(B_r)}{\mu(B_R)} 
    = \frac{\mu(B_r)}{\mu(B_{\alp_j})}\frac{\mu(B_{\alp_j})}{\mu(B_{\alp_{k+1}})}
     \frac{\mu(B_{\alp_{k+1}})}{\mu(B_R)}
     \simle \biggl(\frac{r}{\alp_j}\biggr)^2 \biggl(\frac{\alp_j}{\alp_{k+1}}\biggr)^2
\biggl(\frac{\alp_{k+1}}{R}\biggr)^2 = \Bigl(\frac{r}{R}\Bigr)^2
\]
and 
\[
  \frac{\mu(B_r)}{\mu(B_R)} 
    = \frac{\mu(B_r)}{\mu(B_{\alp_j})}\frac{\mu(B_{\alp_j})}{\mu(B_{\alp_{k+1}})}
     \frac{\mu(B_{\alp_{k+1}})}{\mu(B_R)}
     \simge \biggl(\frac{r}{\alp_j}\biggr)^4 \biggl(\frac{\alp_j}{\alp_{k+1}}\biggr)^4
\biggl(\frac{\alp_{k+1}}{R}\biggr)^4 = \Bigl(\frac{r}{R}\Bigr)^4,
\]
which together with~\eqref{eq-ex1-7}
show that 
\[
\lQ=\lQo=(0,2]
    \quad \text{and} \quad
\uQ=\uQo=[4,\infty).
\]
(The estimates for balls with radius larger than $\al_0=\frac{1}{2}$ are
easier.)
\end{example}

The following example is a modification of Example~\ref{ex1}. It 
shows that we can have
$\sup\lSo=\inf\uSo$ while $\lSo\ne\lQo$ and $\uSo\ne\uQo$.
In this case the common borderline exponent 
of the $S$-sets
belongs to $\lSo$ but not to $\uSo$, thus 
demonstrating the sharpness of Lemma~\ref{lem-1}.

\begin{example}   \label{ex-S-touch}
Consider $\R^2$ and $x=0$.
Let $\alp_k$ and $w$ be as in Example~\ref{ex1}.
Also let $\ga_k=\alp_{k+1} \log k$
and $\de_k=\alp_{k+1} \log^2 k$, $k=3,4,\ldots$\,,
so that $\alp_{k+1} < \ga_k < \de_k < \alp_k$,
and let
\[
      w_2(\rho)=\begin{cases}
                     \alp_{k+1}, & \text{if } \alp_{k+1} \le \rho \le \ga_k, 
                     \ k=3,4,\ldots,\\
                  \rho^2/\de_k,
              & \text{if } \ga_k \le \rho \le \de_{k}, 
                     \ k=3,4,\ldots,\\
                     \rho, & \text{otherwise},
        \end{cases}
\] 
and $d\mu(y) =w_2(|y|)\,dx$.
It follows from Proposition~\ref{prop-suff-admiss-w}
that $w_2$ is $1$-admissible,
as
\[  
     0 \le \frac{\rho w_2'(\rho)}{w_2(\rho)} \le 2
     \quad \text{a.e.}
\]
Since $w(\rho) \le w_2(\rho) \le \rho$ for $\rho \le \alp_2$ we see that
$\mu(B_{\alp_k}) \simeq \alp_k^3$ and $\lSo=(0,3]$.
Moreover,
\[
         \mu(B_{\ga_k} \setm B_{\alp_{k+1}}) 
          \simeq \int_{\alp_{k+1}}^{\ga_k} w(\rho) \rho \, d\rho 
          = \frac{\alp_{k+1}}{2} (\ga_k^2-\alp_{k+1}^2)
          \simeq \alp_k^2 \ga_k^2
          = \alp_{k}^6 \log^2 k 
\]
and
\[
        \mu(B_{\de_k} \setm B_{\ga_k})
        \simeq \int_{\ga_{k}}^{\de_k} \frac{\rho^2}{\de_k}  \rho \,d\rho 
          = \frac{\de_k^4 - \ga_k^4}{4\de_k}
          \simeq \de_k^3.
\]
It follows that 
\[ 
       \mu(B_{\ga_k} ) \simeq \alp_{k}^6 \log^2 k = \frac{\ga_k^3}{\log k}
       \quad \text{and} \quad
       \mu(B_{\de_k} ) \simeq \de_{k}^3.
\]
As in Example~\ref{ex1} one can show that these are the extreme cases, and thus
letting $k\to\infty$ shows that $\uSo=(3,\infty)$.
Moreover,
\[
     \frac{\mu(B_{\alp_{k+1}})}{\mu(B_{\ga_k})}
     \simeq \frac{1}{\log^2 k}
     = \biggl(\frac{\alp_{k+1}}{\ga_k}\biggr)^2.
\]
Since $\alp_{k+1} / \ga_k = 1/{\log k} \to 0$, as $k \to \infty$,
this shows that $p \notin \lQo$ if $p>2$.
As this is the extreme case, we see that $\lQ=\lQo=(0,2]$.
Finally, 
\[
  \frac{\mu(B_{\ga_k})}{\mu(B_{\de_{k}})}
     \simeq \biggl(\frac{\ga_k}{\de_{k}}\biggr)^3 \frac{1}{\log k}
     = \biggl(\frac{\ga_k}{\de_{k}}\biggr)^4,
\]   
which shows that $\uQ=\uQo=[4,\infty)$.
\end{example}

There is nothing special about the end points $2$, $3$, $\frac{10}{3}$ and $4$ 
(or the plane $\R^2$) in 
Example~\ref{ex1}. Indeed, in
the following example we  
indicate how one can construct 
a $1$-admissible weight $w$ in $\R^n$, $n\ge2$, such that
\begin{equation} \label{eq-ex-abcd-alt}
\lQo=(0,a], \quad
\lSo=(0,b], \quad
\uSo=[c,\infty)
    \quad \text{and} \quad
\uQo=[d,\infty),
\end{equation}
where $1<a<b<c<d$.
The reason for the condition $a>1$ 
is that we want
to obtain the $1$-admissibility of $w$ using Proposition~\ref{prop-suff-admiss-w},
see Remark~\ref{rmk-cor-w}.

\begin{example} \label{ex-abcd-alt}
For $1<a<b<c<d$ let
\[
    \la =\frac{(c-a)(d-b)}{(b-a)(d-c)}
\]
and 
\[
   \alp_k = 2^{-\la^k} \text{ and } 
\be_k = \al_k^{(d-b)/(d-c)} = \al_{k+1}^{(c-a)/(b-a)},
 \quad  k=0,1,2,\ldots.
\]
Note that $\la>1$ and thus $\al_k\to0$ as $k\to\infty$.
Also, $\al_{k+1}\ll\be_k\ll\al_k$.
Then the weight
\[
      w(\rho)=\begin{cases}
             \be_{k}^{c-a}\rho^{a-n} = \alp_{k+1}^{b-a}\rho^{a-n}, 
  & \text{if } \alp_{k+1} \le \rho \le \be_k, 
                     \ k=0,1,2,\ldots,\\
             \be_{k}^{c-d}\rho^{d-n} = \alp_{k}^{b-d}\rho^{d-n}, 
  & \text{if } \be_k \le \rho \le \alp_{k}, 
                     \ k=0,1,2,\ldots,\\
                     \alp_0, & \text{if } \rho \ge \alp_0,
        \end{cases}
\] 
is continuous and $1$-admissible on $\R^n$.
Without going into details, 
one then argues similarly to Example~\ref{ex1} to show that 
\eqref{eq-ex-abcd-alt} holds.
\end{example}

\section{Background results on metric spaces}
\label{sect-metric}

In this section we are going to introduce the
necessary background on Sobolev spaces and capacities in metric spaces.
Proofs of most of the results mentioned in 
the first half of this section can be found in the monographs
Bj\"orn--Bj\"orn~\cite{BBbook} and
Heinonen--Koskela--Shanmugalingam--Tyson~\cite{HKSTbook}.
Towards the end
of this section we obtain some new results.

We begin with the notion of upper gradients 
as defined by Heinonen and Kos\-ke\-la~\cite{HeKo98}
(who called them very weak gradients).

\begin{deff} \label{deff-ug}
A nonnegative Borel function $g$ on $X$ is an \emph{upper gradient} 
of an extended real-valued function $f$
on $X$ if for all (nonconstant, compact
and rectifiable) curves  
$\gamma \colon [0,l_{\gamma}] \to X$,
\begin{equation} \label{ug-cond}
        |f(\gamma(0)) - f(\gamma(l_{\gamma}))| \le \int_{\gamma} g\,ds,
\end{equation}
where we follow the convention that the left-hand side is $\infty$ 
whenever at least one of the 
terms therein is infinite.
If $g$ is a nonnegative measurable function on $X$
and if (\ref{ug-cond}) holds for \p-almost every curve (see below), 
then $g$ is a \emph{\p-weak upper gradient} of~$f$. 
\end{deff}

A \emph{curve} is a continuous mapping from an interval,
and a \emph{rectifiable} curve is a curve with finite length.
We will only consider curves which are nonconstant, compact
and 
rectifiable, and thus each curve can 
be parameterized by its arc length $ds$. 
A property is said to hold for \emph{\p-almost every curve}
if it fails only for a curve family $\Ga$ with zero \p-modulus, 
i.e.\ there exists $0\le\rho\in L^p(X)$ such that 
$\int_\ga \rho\,ds=\infty$ for every curve $\ga\in\Ga$.
Note that a \p-weak upper gradient \emph{need not} be a Borel function,
it is only required to be measurable.
On the other hand,
every measurable function $g$ can be modified on a set of measure zero
to obtain a Borel function, from which it follows that 
$\int_{\gamma} g\,ds$ is defined (with a value in $[0,\infty]$) for \p-almost every  
curve $\ga$. 

The \p-weak upper gradients were introduced in
Koskela--MacManus~\cite{KoMc}. It was also shown there
that if $g \in \Lp(X)$ is a \p-weak upper gradient of $f$,
then one can find a sequence $\{g_j\}_{j=1}^\infty$
of upper gradients of $f$ such that $g_j \to g$ in $L^p(X)$.
If $f$ has an upper gradient in $\Lp(X)$, then
it has a \emph{minimal \p-weak upper gradient} $g_f \in \Lp(X)$
in the sense that for every \p-weak upper gradient $g \in \Lp(X)$ of $f$ we have
$g_f \le g$ a.e., see Shan\-mu\-ga\-lin\-gam~\cite{Sh-harm}
and Haj\l asz~\cite{Haj03}. The minimal \p-weak upper gradient is well defined
up to a set of measure zero in the cone of nonnegative functions in $\Lp(X)$.
Following Shanmugalingam~\cite{Sh-rev}, 
we define a version of Sobolev spaces on the metric measure space $X$.

\begin{deff} \label{deff-Np}
For a measurable function $f:X\to\R^n$, let 
\[
        \|f\|_{\Np(X)} = \biggl( \int_X |f|^p \, \dmu 
                + \inf_g  \int_X g^p \, \dmu \biggr)^{1/p},
\]
where the infimum is taken over all upper gradients of $f$.
The \emph{Newtonian space} on $X$ is 
\[
        \Np (X) = \{f: \|f\|_{\Np(X)} <\infty \}.
\]
\end{deff}
\medskip

The space $\Np(X)/{\sim}$, where  $f \sim h$ if and only if $\|f-h\|_{\Np(X)}=0$,
is a Banach space and a lattice, see Shan\-mu\-ga\-lin\-gam~\cite{Sh-rev}.
In this paper we assume that functions in $\Np(X)$ are defined everywhere,
not just up to an equivalence class in the corresponding function space.
This is needed for the definition of upper gradients to make sense.
For a measurable set $E\subset X$, the Newtonian space $\Np(E)$ is defined by
considering $(E,d|_E,\mu|_E)$ as a metric space on its own. 
If $f,h \in \Nploc(X)$, then $g_f=g_h$ a.e.\ in $\{x \in X : f(x)=h(x)\}$,
in particular $g_{\min\{f,c\}}=g_f \chi_{f < c}$ for $c \in \R$.

\begin{deff}
The \emph{Sobolev \p-capacity} of an arbitrary set $E\subset X$ is
\[
\Cp(E) = \inf_u\|u\|_{\Np(X)}^p,
\]
where the infimum is taken over all $u \in \Np(X)$ such that
$u\geq 1$ on $E$.
\end{deff}

The Sobolev capacity is countably subadditive. 
We say that a property holds \emph{quasieverywhere} (q.e.)\ 
if the set of points  for which it fails 
has Sobolev capacity zero. 
The Sobolev  capacity is the correct gauge 
for distinguishing between two Newtonian functions. 
If $u \in \Np(X)$, then $u \sim v$ if and only if $u=v$ q.e.
Moreover, Corollary~3.3 in Shan\-mu\-ga\-lin\-gam~\cite{Sh-rev} shows that 
if $u,v \in \Np(X)$ and $u= v$ a.e., then $u=v$ q.e.
This is the main reason why, unlike in the classical Euclidean setting, 
we do not need to 
require the functions admissible in the definition of capacity to be $1$ in a 
neighbourhood of $E$. 
Theorem~4.5 in~\cite{Sh-rev} shows that for open $\Om\subset\R^n$, 
the quotient space $\Np(\Om)/{\sim}$ 
coincides with the usual Sobolev space $W^{1,p}(\Om)$. 
For weighted $\R^n$, the corresponding results can be found in 
Bj\"orn--Bj\"orn~\cite[Appendix~A.2]{BBbook}.
It can also be shown that in this case
$\Cp$ is the usual Sobolev capacity in (weighted
or unweighted) $\R^n$.

\begin{deff} \label{def-PI}
We say that $X$ supports a \emph{\p-Poincar\'e inequality at $x$} if
there exist constants $C>0$ and $\lambda \ge 1$
such that for all balls $B=B(x,r)$, 
all integrable functions $f$ on $X$, and all upper gradients $g$ of $f$, 
\[ 
        \vint_{B} |f-f_B| \,\dmu
        \le C r \biggl( \vint_{\lambda B} g^{p} \,\dmu \biggr)^{1/p},
\] 
where $ f_B 
 :=\vint_B f \,\dmu 
:= \int_B f\, d\mu/\mu(B)$.
If $C$ and $\lambda$ are independent of $x$, we say that
$X$ supports a \emph{\textup{(}global\/\textup{)} \p-Poincar\'e inequality}.
\end{deff}

In the definition of Poincar\'e inequality we can equivalently assume
that $g$ is a \p-weak upper gradient---see the comments above.
It was shown by Keith--Zhong~\cite{KeZh} that if $X$ is complete and $\mu$ 
is globally
doubling and supports a global \p-Poincar\'e inequality 
with $p>1$,  then $\mu$ actually
supports a global $p_0$-Poincar\'e inequality for some $p_0<p$. 
The completeness of $X$ is needed for Keith--Zhong's result, as shown by
Koskela~\cite{Koskela}.
In some of our  estimates we will need such a better $p_0$-Poincar\'e inequality
at~$x$, which (by Koskela's example) does not follow from the \p-Poincar\'e inequality
at~$x$.

If $X$ is complete and  $\mu$ is globally doubling
and supports a global \p-Poincar\'e inequality,
then 
the functions in $\Np(X)$
and those in $\Np(\Om)$, for open 
$\Om \subset X$, 
are \emph{quasicontinuous},
see Bj\"orn--Bj\"orn--Shan\-mu\-ga\-lin\-gam~\cite{BBS5}.
This means that in the Euclidean setting 
$\Np(\R^n)$ and $\Np(\Om)$ are the 
refined Sobolev spaces as defined in
Heinonen--Kilpel\"ainen--Martio~\cite[p.~96]{HeKiMa},
see Bj\"orn--Bj\"orn~\cite[Appendix~A.2]{BBbook} 
for a proof of this fact valid in weighted $\R^n$.
 
To be able to define the variational capacity we first
need a Newtonian space with zero boundary values.
We let, for an open set $\Om \subset X$,
\[
\Np_0(\Om)=\{f|_{\Om} : f \in \Np(X) \text{ and }
        f=0 \text{ on } X \setm \Om\}.
\]

\begin{deff}
Let $\Om\subset X$ be open. The \emph{variational 
\p-capacity} of $E\subset \Om$ with respect to $\Om$ is
\[
\cp(E,\Om) = \inf_u\int_{\Om} g_u^p\, d\mu,
\]
where the infimum is taken over all $u \in \Np_0(\Om)$
such that
$u\geq 1$ on $E$.
\end{deff}

Also the 
variational capacity 
is countably subadditive and coincides with the usual variational
capacity in the case when $\Om\subset\R^n$ is open
(see Bj\"orn--Bj\"orn~\cite[Theorem~5.1]{BBvarcap} for a proof
valid in weighted $\R^n$). 
We are next going to establish three new results
concerning the variational capacity.
Propositions~\ref{prop-Cp=0<=>cp=0} and~\ref{prop-zero-cap-capacitable}
will only be used in Proposition~\ref{prop:zero cap}
(and Example~\ref{ex-log-weights})
to prove a condition for a point to have positive capacity, 
while Proposition~\ref{lem-cpX} will only be used for proving
Propositions~\ref{prop:S-infty}
and~\ref{prop-cap-formula-radial}
(and in Example~\ref{ex-log-weights}), which 
deal with the
variational capacity taken with respect to the whole space. 
These results may also 
be of independent interest.

It is well known that if $X$ supports a global
$(p,p)$-Poincar\'e inequality
(i.e.\
a Poincar\'e inequality with an $L^p$ norm  instead of an $L^1$ norm
in the left-hand side),
then the variational and Sobolev capacities have
the same zero sets (if $\Om$ is bounded and $\Cp(X \setm \Om)>0$).
We will need the following generalization of this fact.
Since we do not have the same tools available, our proof
is different and more direct than those in the literature.
Note also that we only require a \p-Poincar\'e inequality (at $x$),
not a $(p,p)$-Poincar\'e inequality.

\begin{prop} \label{prop-Cp=0<=>cp=0}
Assume that $X$ supports a \p-Poincar\'e inequality at some $x\in X$,
that\/ $\Om$ is a bounded open set,
and that $E \subset \Om$.
Then\/ $\cp(E,\Om)=0$ if and only if $\Cp(E)=0$ or $\Cp(X\setm \Om)=0$.
\end{prop}

The Poincar\'e assumption cannot be completely omitted, 
as is easily seen
by considering a nonconnected example, or a bounded ``bow-tie''
as in Example~5.5 in Bj\"orn--Bj\"orn~\cite{BBnonopen}.
However,
we actually do not need the full \p-Poincar\'e inequality
at $x$, since it is enough to have a \p-Poincar\'e inequality
for some large enough ball $B$ 
(i.e.\ such that $\Om \subset B$ and $\Cp(B \setm \Om)>0$).
This somewhat resembles 
the situation concerning
Friedrichs' inequality (also called Poincar\'e inequality for $\Np_0$)
and its role in the uniqueness of minimizers,
see the discussion in Section~5 in Bj\"orn--Bj\"orn~\cite{BBnonopen}.
For an easy example of a space which supports a Poincar\'e inequality for 
large balls but not for small balls, see Example~5.9 in \cite{BBnonopen}.

\begin{proof}
If $\Cp(E)=0$, then $u:=\chi_E \in \Np_0(\Om)$,
while if $\Cp(X \setm \Om)=0$, then $u:=\chi_\Om \in \Np_0(\Om)$. 
In both cases this yields that $\cp(E,\Om) \le \int_\Om g_u^p\,d\mu=0$.

Conversely, assume that $\cp(E,\Om)=0$ and that $\Cp(X \setm \Om)>0$.
We need to show that $\Cp(E)=0$.
Choose a ball $B$ centred at $x$ and containing $\Om$ such that $\Cp(B \setm \Om)>0$.
By Lemma~2.24 in Bj\"orn--Bj\"orn~\cite{BBbook}, also $\CpB(B \setm \Om)>0$,
where $\CpB$ is the Sobolev capacity 
with respect to the ambient space $B$.
Let $0 \le u \le 1$ be admissible for $\cp(E,\Om)$.
Then 
\[
\mu\bigl(\bigl\{y\in B:u(y)\le\tfrac12\bigr\}\bigr)\ge\tfrac12 \mu(B)
\quad \text{or} \quad
\mu\bigl(\bigl\{y\in B:u(y)\ge\tfrac12\bigr\}\bigr)\ge\tfrac12 \mu(B).
\]
In the former case we let $v=(2u-1)_\limplus:=\max\{2u-1,0\}$, 
while in the latter we let
$v=(1-2u)_\limplus$.
In both cases  $g_v\le2g_u$ and $\mu(A)\ge\tfrac12\mu(B)$, where
$A=\{y\in B:v(y)=0\}$.
Since $v_B=|v-v_B|$ 
in $A$, we have by the \p-Poincar\'e inequality for $B$ that
\[
    v_B = \vint_A |v-v_B|\,d\mu \le 2 \vint_B |v-v_B|\,d\mu 
       \simle \biggl(\int_B g_v^p \, d\mu\biggr)^{1/p}.
\]
Hence, as $0\le v\le1$   and $g_v\le2g_u$, we have
\begin{align*}
\CpB(\{y\in B:v(y)=1\}) &\le \int_B( v^p + g_v^p)\,d\mu
\le \int_B v \, d\mu + \int_B g_v^p \, d\mu \\
&= \mu(B) v_B + \int_B g_v^p \, d\mu 
\simle \biggl(\int_B g_u^p \, d\mu\biggr)^{1/p} + \int_B g_u^p \, d\mu,
\end{align*}
where the implicit constant in $\simle$ depends on $B$ but is independent of 
$u$. 
Taking infimum over all admissible $u$ shows that, 
depending on the choices of $v$, 
we have at least one of 
$\CpB(E)=0$ and
 $\CpB(B\setm\Om)=0$, the latter being impossible by the choice of $B$.
Thus $\CpB(E)=0$ and Lemma~2.24 in \cite{BBbook} completes the proof.
\end{proof}

If $X$ is  complete and  $\mu$ is globally doubling
and supports a global \p-Poincar\'e inequality,
then it is known that the variational capacity is an outer capacity,
i.e.\ if $E$ is a compact subset of $\Omega$
 then 
\[ 
\cp(E,\Om)=\inf_{\substack{G \text{ open} \\  E\subset G \subset \Om}} \cp(G,\Om),
\] 
see Bj\"orn--Bj\"orn--Shan\-mu\-ga\-lin\-gam~\cite[p.\ 1199]{BBS5}
and Theorem~6.19 in Bj\"orn--Bj\"orn~\cite{BBbook}.
We will need a version of this result for sets of 
zero capacity under our more general assumptions.
For the Sobolev capacity such a result was obtained
in~\cite{BBS5}, Proposition~1.4 
(which can also be found as Proposition~5.27 in \cite{BBbook}), 
under the assumption that $X$ is proper.
(Recall that a metric space $X$ is \emph{proper} if  
all closed bounded
subsets are compact. If $\mu$ is globally doubling,
then $X$ is proper if and only if $X$ is complete.)
A modification of that proof yields the following generalization,
which only requires local compactness 
near 
$E$ and at the same time
also gives the conclusion for the variational capacity.
This generalization
was partly inspired by the discussion of the corresponding
result in
Heinonen--Koskela--Shanmugalingam--Tyson~\cite{HKSTbook}.
In combination with Proposition~\ref{prop-Cp=0<=>cp=0},
Proposition~\ref{prop-zero-cap-capacitable} 
gives the outer capacity property for sets of 
zero variational capacity under very mild assumptions.

\begin{prop}\label{prop-zero-cap-capacitable}
Let\/ $\Om$ be an open set,
and let $E\subset \Om$ with $\Cp(E)=0$.
Assume that there is a locally compact open set $G \supset E$.
Then, for every $\eps>0$, there is  an open set $U\supset E$ with
\[
   \cp(U,\Om)<\eps \quad \text{and} \quad \Cp(U) < \eps.
\]
\end{prop}

We outline the main ideas of the proof,
see the above references for more details.

\begin{proof}[Sketch of proof.]
First assume that $\itoverline{G}$ is compact, 
and choose a bounded open set $V\supset E$ such that 
$V \subset G \cap \Om$ and $\int_V(\rho+1)^p\,d\mu<\eps$, where $\rho$ is a
lower semicontinuous upper gradient of $\chi_E\in\Np(X)$.
The function $u(x):=\min\{1,\inf_\ga \int_\ga(\rho+1)\,ds\}$, with the infimum taken 
over all curves connecting $x$ to $X\setm V$
(including constant curves), 
has $(\rho+1)\chi_V$ as an upper gradient,
and $u=1$ in $E$.
Lemma~3.3 in \cite{BBS5} 
shows that $u$ is lower semicontinuous in $G$ and hence everywhere,
since $u=0$ in $X \setm V$ by construction.
This also shows that $u \in \Np_0(\Om)$.
Using $u$ as a test function for the level set $U:=\{x:
u(x)>\tfrac12\}$ 
shows that
$\cp(U,\Om)\simle\eps$ and $\Cp(U)\simle\eps$, and proves the claim 
in this case.

If $G$ is merely locally compact, we use separability
to find a suitable countable cover of $E$, and then
conclude the result using the countable subadditivity of the 
capacities.
\end{proof}

A direct consequence of Proposition~\ref{prop-zero-cap-capacitable}
is that the assumption
that $X$ is proper can be replaced by the assumption that
$\Omega$ is locally compact in
Theorem~5.29 and Propositions~5.28 and~5.33 in
Bj\"orn--Bj\"orn~\cite{BBbook}, see also 
Bj\"orn--Bj\"orn--Shan\-mu\-ga\-lin\-gam~\cite{BBS5} and 
Heinonen--Koskela--Shanmugalingam--Tyson~\cite{HKSTbook}.

We will also need the following result.

\begin{lem} \label{lem-cpX}
Let $E \subset X$ be bounded and let $x\in X$.
Then 
\[ 
    \cp(E,X)=\lim_{r \to \infty} \cp(E,B(x,r)).
\]
\end{lem}

\begin{proof}
That $\cp(E,X)\le \lim_{r \to \infty} \cp(E,B(x,r))$ is trivial.
To prove the converse, we may assume that $\cp(E,X) < \infty$.
Let $\eps >0$ and let $u$ be admissible
for $\cp(E,X)$ and such that $\int_X g_u^p \, d\mu < \cp(E,X) + \eps$. 
Then $u_n:=u\eta_n\to u$ in $\Np(X)$, as $n \to \infty$, where 
$\eta_n(y)=(1-\dist(y,B(x,n))_\limplus$.
Hence,
\[
\lim_{n \to \infty} \cp(E,B(x,2n)) 
     \le \lim_{n \to \infty} \int_X g_{u_n}^p\,d\mu
     \le \cp(E,X) + \eps.
\]
Letting $\eps \to 0$ concludes the proof.
\end{proof}

Our lower bound estimates for the capacities are all based on the following
telescoping argument, which is well-known under the assumptions
that $\mu$ is globally doubling and supports a
global \p-Poincar\'e inequality. 
However, it is enough to require the \p-Poincar\'e inequality, as well as 
the doubling and reverse-doubling
conditions, at $x$ only.
We therefore recall the short proof.

\begin{lem}\label{lem:chain estimate}
Assume that $\mu$ is doubling and reverse-doubling at $x$ and 
supports a \p-Poincar\'e inequality at $x$.
Let\/ 
$0<r < R \le \diam X /2\tau$,
where $\tau>1$ is the constant from the reverse-doubling 
condition~\eqref{eq:rev-doubling-x0}. 
Write $r_k=2^k r$ and $B^k=B(x,r_k)$ for $k\in\Z$, and let $k_0$ be such that
$r_{k_0}
\leq R < r_{k_0+1}$. 
Then for any $u\in N_0^{1,p}(B_{R})$
we have
\begin{equation}\label{eq:chain estimte}
|u_{B_r}| \simle \sum_{k=1}^{k_0+1} r_k \biggl(\vint_{\lambda B^k} g_u^p \,d\mu\biggr)^{1/p},
\end{equation}
where $\la$ is the dilation constant in the \p-Poincar\'e inequality at $x$.
\end{lem}

\begin{proof}
For $u\in N_0^{1,p}(B_{R})$ we have $u_A=0$, where $A={B_{\tau R}\setm B_R}$.
Let $B^* = B_{\tau R}\cup B_{2R}$. 
Then 
\begin{align*}
|u_{B_r}| & \le |u_{B_r}-u_{B^{k_0+1}}| + |u_{B^{k_0+1}}-u_{A}| \\
  & \le \sum_{k=1}^{k_0+1}|u_{B^{k}}-u_{B^{k-1}}| + |u_{B^{k_0+1}}-u_{B^*}| 
      + |u_{A}-u_{B^*}|.
\end{align*}
Since $\mu$ is doubling and reverse-doubling at $x$, it is easy to verify that 
\[
\mu(A)\simeq\mu(B_{\tau R})\simeq\mu(B^*)\simeq\mu(B^{k_0+1}).
\]
The doubling condition and \p-Poincar\'e inequality at $x$,
together with the fact that $B^{k_0+1}\subset B^*$ and $A\subset B^*$,
then show that
\begin{align*}
|u_{B_r}| 
  & \simle \sum_{k=1}^{k_0+1} \vint_{B^{k}}|u-u_{B^{k}}|\,d\mu  
         +  \vint_{B^*}|u-u_{B^*}|\,d\mu\\
  & \simle \sum_{k=1}^{k_0+1} r_k \biggl(\vint_{\lambda B^k} g_u^p \,d\mu\biggr)^{1/p}
         + R\biggl(\vint_{\lambda B^*} g_u^p \,d\mu\biggr)^{1/p}.
\end{align*}
The claim follows, since the last integral  is comparable to 
$\vint_{\lambda B^{k_0+1}} g_u^p \,d\mu$.  
\end{proof}

\begin{remark} \label{rmk-Q-vs-Q0}
In the forthcoming sections we give several different capacity
estimates involving the exponent sets $\lQ$ and $\uQ$. 
In these results
(and in Lemma~\ref{lem:chain estimate} above),
the implicit constants in $\simeq$, $\simle$
and $\simge$
will always be independent of $r$ and $R$, but they 
may depend on $x$, $X$, $\mu$, $p$
and (the auxiliary exponent) $q$. 
The dependence on $x$, $X$
and $\mu$ will only be through the constants
in the doubling, reverse-doubling
and Poincar\'e assumptions, as well as through
the constants $C_q$ in the definitions of
the $Q$-sets. 
In particular, if these conditions hold in all of $X$
with uniform constants, then we obtain capacity estimates which are independent of $x$ as well.

There are also corresponding 
estimates involving $\lQo$, $\lQi$, 
$\uQo$ and $\uQi$, which are just easy reformulations with appropriate restrictions
on the radii, viz.\ $R\le R_0$ for the $\lQo$- and $\uQo$-sets, and $r\ge R_0$
for the $\lQi$- and $\uQi$-sets, where $0<R_0<\infty$ is 
fixed, cf.\ 
Theorems~\ref{thm:intro} and~\ref{thm:borderline-intro}.
In these restricted estimates, 
as well as in the estimates in Section~\ref{sect-S} 
involving the $S$-sets, the
implicit constants in $\simeq$, $\simle$ and $\simge$ will 
in addition depend on $R_0$.
Observe also that, by e.g.\ Lemmas~\ref{lem-3} and~\ref{lem-2},
the exponent sets are independent of $R_0$, but the
constants $C_q$ do depend on the range of radii.

For these restricted estimates one can also weaken the assumptions a little:
The doubling and reverse-doubling conditions and the Poincar\'e inequality 
are only needed for balls with radii 
in the considered range, i.e.\ for $r\le R_0\max\{2,\tau\}$
or for $r\ge 2R_0$.
Arguing as in Lemma~\ref{lem-2},
it is easily seen that in the case of the doubling condition (but not for 
reverse-doubling and the Poincar\'e inequality) this is equivalent to 
assuming doubling for all $r\le1$ or $r\ge1$, respectively.
For the reverse-doubling and the Poincar\'e inequality, the range of radii for which they hold
is however essential, as can be seen by e.g.\ 
letting $X$ be the union of two disjoint closed balls in $\R^n$.

The factor 2 in the above bounds on radii is only dictated by the dyadic balls in
the proof of Lemma~\ref{lem:chain estimate} and can equivalently be replaced by any 
$\s>1$, upon correspondingly changing the choice of balls therein.
Again, this will be reflected in the implicit constants. 
\end{remark}

\section{Upper bounds for capacity}
\label{sect-5}

\emph{From now on we make the general assumption that
$\mu$ is doubling at $x$.
Recall also that\/ $1\le p<\infty$.}

\medskip

The following simple upper bound for capacity
is valid for any $1\le p < \infty$. 
Note that we do not need any Poincar\'e inequality
(nor reverse-doubling)
to obtain any of our upper bound estimates.

\begin{prop}\label{prop:simple upper bound}
Let\/  $0<2r\leq R$. Then 
\[
\cp(B_r,B_{R}) \simle \min\biggl\{\frac{\mu(B_r)}{r^{p}},\frac{\mu(B_R)}{R^{p}}\biggr\}.
\]
\end{prop}

For $p\in\lQ$ (resp.\ $p\in\uQ$), the first (resp.\ second)
term in the minimum 
gives the sharper estimate, but for
$p$ in between the $Q$-sets the minimum can vary depending on the 
radii, as can be seen in Example~\ref{ex1-partb}. 
See Section~\ref{sect:interior} for corresponding lower estimates.

It is essential to bound $r$ away from $R$ in Proposition~\ref{prop:simple upper bound} 
 since in general
 $\cp(B_r,B_{R})\to \infty$ as $r\to R$. This is apparent and well-known in
 unweighted $\R^n$ (cf.\ Example~2.12
 in Heinonen--Kilpel\"ainen--Martio~\cite{HeKiMa}), 
 but similar behaviour is present in more general metric spaces as well. 
 (This restriction should thus be taken into account in the upper bounds in~\cite{CaDaGa2} 
 and~\cite{GaMa} as well.)
Capacity of thin annuli (with $R/2<r<R$) in the metric setting is studied
in~\cite{BBL-thin}.

\begin{proof}
 Take 
\[u_r(y)=\biggl(1-\frac{\dist(y,B_r)}{r}\biggr)_\limplus\quad \text{and} \quad 
  u_R(y)=\biggl(1-\frac{\dist(y,B_{R/2})}{R/2}\biggr)_\limplus.\]
  Both of these are admissible for $\cp(B_r,B_{R})$, and clearly
(by doubling),
\[
\int_{B_{R}} g_{u_r}^p \,d\mu \le \frac{\mu(B_{2r})}{r^{p}} 
\simle \frac{\mu(B_r)}{r^{p}}
\quad \text{and} \quad
\int_{B_{R}} g_{u_R}^p \,d\mu \le \frac{\mu(B_{R})}{(R/2)^{p}} 
\simle \frac{\mu(B_{R})}{R^{p}}. 
\qedhere
\]
\end{proof}

The 
following logarithmic upper bounds  
are particularly useful in the borderline cases $p=\max\lQ$ and $p=\min\uQ$. 
These estimates 
are valid also for $p=1$, as well as for $p\in\interior\lQ$ and 
$p\in\interior\uQ$, but in these cases 
 Proposition~\ref{prop:simple upper bound} actually gives better upper bounds 
 for $\cp(B_r,B_R)$.
Note also that even for the borderline cases $p=\max\lQ$ and $p=\min\uQ$,
the estimates in Proposition~\ref{prop:simple upper bound} can be sharp,
and better than those in Proposition~\ref{prop:upper bounds
  borderline} below,
as shown at the end of Example~\ref{ex1-partb}.

\begin{prop}\label{prop:upper bounds borderline}
Let\/ $0<2r\le R$. 

\begin{enumerate}
\item \label{it3-a}
If $p\in \lQ$, 
then
\begin{equation} \label{eq-it3-a}
\cp(B_r,B_{R})\simle \frac{\mu(B_{R})}{R^{p}}\biggl(\log\frac R r\biggr)^{1-p}.
\end{equation}

\item \label{it3-b}
 If $p\in \uQ$, 
then
\begin{equation} \label{eq-it3-b}
\cp(B_r,B_{R})\simle \frac{\mu(B_r)}{r^{p}}\biggl(\log\frac R r\biggr)^{1-p}.
\end{equation}
\end{enumerate}
\end{prop}

{Examples~\ref{ex-log-weights}\,(b) and
\ref{ex-log-weights-large}\,(b) show that these estimates are sharp.}

\begin{proof}
Choose 
\[
u(y)=\min\biggl\{1,\frac{\log(R/d(y,x))}{\log(R/r)}\biggr\}_\limplus
\quad  \text{and} \quad
g(y) = \frac{\chi_{B_R \setm B_r}}{\log(R/r)  d(y,x)}.
\]
Then $u$ is admissible for  $\cp(B_r,B_{R})$, 
and $g$ is a \p-weak upper gradient of $u$,
by Theorem~2.16  in Bj\"orn--Bj\"orn~\cite{BBbook}.
Write $r_k=2^k r$ and $B^k=B(x,r_k)$, and let $k_0\in\Z$ be such that 
$r_{k_0}\le R < r_{k_0+1}$.
Then
\begin{equation}\label{eq:up border}\begin{split}
 \cp(B_r,B_{R}) & \leq \int_{B_{R}} g^p\, \dmu \leq 
  \sum_{k=1}^{k_0+1} \int_{B^k\setminus B^{k-1}} g^p\,\dmu
   \simle \frac{1}{\log^p(R/r)} \sum_{k=1}^{k_0+1} \frac{\mu(B^k)}{r_k^{p}}.
\end{split}
\end{equation}
For 
$p\in\lQ$ we have that $r_k^{-p}\mu(B^k)\simle R^{-p}\mu(B_{R})$ 
when $1 \le k \le k_0+1$,
and for $p\in\uQ$ that $r_k^{-p}\mu(B^k)\simle r^{-p}\mu(B_r)$ 
for all $k \ge 1$.
Since $0<r\le R/2$, we have $k_0+1\simle\log(R/r)$, 
and so both claims follow from \eqref{eq:up border}. 
\end{proof}

\section{Lower bounds for capacity}\label{sect:interior}

The results in this section complement the upper
bounds in Section~\ref{sect-5}, and 
for $p$ in the interior of (one of)
the $Q$-sets these together yield 
the sharp estimates announced in Theorem~\ref{thm:intro}. 
For $p$ in between the $Q$-sets, the lower and upper bounds do not meet, but we shall
see in Proposition~\ref{prop:low bounds beyond borderline} that the lower 
bounds indicate the distance from $p$ to the corresponding $Q$-set.
Example~\ref{ex1-partb} shows that in this case both the upper bounds in 
Proposition~\ref{prop:simple upper bound} and the lower 
bounds~\eqref{eq:low bounds beyond borderline} and~\eqref{eq:up bounds beyond borderline}
in Proposition~\ref{prop:low bounds beyond borderline}
are optimal.
See also Proposition~\ref{prop-sharp}, which further 
demonstrates the sharpness of these estimates.

Also note that for the lower bounds without logarithmic terms we do not need the
restriction $2r\le R$, since the capacity of thin annuli is minorized by the capacity
of thick annuli.
In the borderline cases, where $\log(R/r)$ plays a role, the restriction
 $2r\le R$ 
is still needed. 
As in Lemma~\ref{lem:chain estimate}, we however require that $R \le \diam X/2\tau$,
where $\tau>1$ is the constant from the reverse-doubling 
condition~\eqref{eq:rev-doubling-x0}. 
See Remark~\ref{rmk-Q-vs-Q0} for comments on how the choice of the involved parameters
influences the implicit constants in $\simeq$, $\simle$ and $\simge$.

\begin{prop}\label{prop:low bounds}
Assume that $\mu$ is reverse-doubling at $x$ and
supports  a \p-Poincar\'e inequality at $x$.
Let\/ $0<r<R \le  \diam X/2\tau$. 
\begin{enumerate}
\item \label{it-1-interior} If $p\in \interior\lQ$, then
\begin{equation} \label{eq-it-1}
\cp(B_r,B_{R})\simge \frac{\mu(B_r)}{r^{p}}.
\end{equation}

\item \label{it-2-interior}
If $p\in \interior\uQ$, then
\begin{equation} \label{eq-it-2}
\cp(B_r,B_{R})\simge \frac{\mu(B_{R})}{R^{p}}.
\end{equation}
\end{enumerate}
\end{prop}

With this we can now prove Theorem~\ref{thm:intro},
which also shows that the estimates in Proposition~\ref{prop:low bounds}
are sharp.

\begin{proof}[Proof of Theorem~\ref{thm:intro}]
Combining Propositions~\ref{prop:simple upper bound} and~\ref{prop:low bounds}
and appealing to Remark~\ref{rmk-Q-vs-Q0}
yield \ref{it-intro-1} and \ref{it-intro-2}.
The last part follows from Proposition~\ref{prop-sharp} below.
\end{proof}

The comparison constants in~\eqref{eq-it-1} and~\eqref{eq-it-2} depend on $p$. In particular,
the constants in our proof tend to zero as $p\nearrow\sup\lQ$ in~\ref{it-1-interior} and
as $p\searrow\inf\uQ$ in~\ref{it-2-interior}. 
This is quite natural, since already unweighted $\R^n$ shows
that these estimates do not always hold when $p=\max \lQ$ and $p=\min \uQ$, 
respectively.
In fact, if $X$ is Ahlfors \p-regular, and thus 
$\lQ=\lSo=\lSi=(0,p]$ and $\uQ=\uSo=\uSi=[p,\infty)$,
Proposition~\ref{prop:upper with S}\,\ref{it-3-upS} 
shows that \eqref{eq-it-1} and \eqref{eq-it-2} fail.
Moreover, Proposition~\ref{prop-sharp}
shows that the estimates 
in Proposition~\ref{prop:low bounds} can never hold 
for all $r$ and $R$  when $p$ is outside of the $Q$-sets.

\begin{proof}[Proof of Proposition~\ref{prop:low bounds}]
Let $u$ be admissible for $\cp(B_r,B_{R})$,
and let $B^k$ be a chain of balls, with radii $r_k$,
 as in Lemma~\ref{lem:chain estimate}.
From Lemma~\ref{lem:chain estimate}
we obtain, for any $0<q <\infty$, that
\begin{align}\label{eq:common}
1 & \simle \sum_{k=1}^{k_0+1} r_k \biggl(\vint_{\lambda B^k} g_u^p \,d\mu\biggr)^{1/p} 
    \le \sum_{k=1}^{k_0+1} \frac {r_k}{\mu(B^k)^{1/p}}
       \biggl(\int_{\lambda B^k} g_u^p \,d\mu\biggr)^{1/p} \nonumber \\
  & \le \biggl(\int_{B_{R}} g_u^p \,d\mu\biggr)^{1/p} 
       \sum_{k=1}^{k_0+1} {\biggl( \frac{r_k^{q}}{\mu(B^k)} \biggr)^{1/p}} r_k^{1-q/p}.
\end{align}
In \ref{it-1-interior} we choose $q>p$ such that $q\in \lQ$, and so we
have for all $1\le k \le k_0+1$ that
\begin{equation}  \label{eq-est-Bk-Br}
\frac{r_k^q}{\mu(B^k)} \simle \frac{r^q}{\mu(B_r)}.
\end{equation}
Since $1-q/p<0$, the sum in the last line of
\eqref{eq:common}
can thus be estimated as 
\begin{align*}
 \sum_{k=1}^{k_0+1} {\biggl( \frac{r_k^{q}}{\mu(B^k)} \biggr)^{1/p}} r_k^{1-q/p}
  & \simle \biggl( \frac{r^{q}}{\mu(B_r)} \biggr)^{1/p} \sum_{k=1}^{k_0+1} r_k^{1-q/p}\\
   & \simle \biggl( \frac{r^{q}}{\mu(B_r)} \biggr)^{1/p} r^{1-q/p}
   = \biggl(\frac{r^p}{\mu(B_r)}\biggr)^{1/p},
\end{align*}
giving 
\[
 \int_{B_{R}} g_u^p \,d\mu \simge \frac{\mu(B_r)}{r^{p}}.
\]
Taking infimum over all admissible $u$
finishes the proof of part \ref{it-1-interior}.

In \ref{it-2-interior} we instead choose $q\in \uQ$ such that $q<p$, 
and so we have for all $1\le k \le k_0+1$ that
\begin{equation}  \label{eq-est-Bk-BR}
\frac{r_k^q}{\mu(B^k)} \simle \frac{R^q}{\mu(B_R)}.
\end{equation}
Now $1-q/p>0$, and thus the sum in the last line of \eqref{eq:common}
can be estimated as 
\begin{align*}
 \sum_{k=1}^{k_0+1} {\biggl( \frac{r_k^{q}}{\mu(B^k)}\biggr)^{1/p}} r_k^{1-q/p}
  & \simle \biggl( \frac{R^{q}}{\mu(B_R)} \biggr)^{1/p} \sum_{k=1}^{k_0+1} r_k^{1-q/p}\\
  & \simle \biggl( \frac{R^{q}}{\mu(B_R)} \biggr)^{1/p} R^{1-q/p}
    = \biggl(\frac{R^p}{\mu(B_R)}\biggr)^{1/p},
\end{align*}
giving 
\[
 \int_{B_{R}} g_u^p \,d\mu \gtrsim \frac{\mu(B_R)}{R^{p}},
\]
and the claim follows by taking infimum over all admissible $u$.
\end{proof}

A modification of the above proof gives the following result, 
which is interesting mainly in the case 
when $p$ is in between the $Q$-sets, i.e.\ $p\notin\lQ\cup\uQ$.

\begin{prop} \label{prop:low bounds beyond borderline}  
Assume that $\mu$ is reverse-doubling at $x$ and
supports  a \p-Poincar\'e inequality at $x$.
Let\/ 
$0<r<R \le \diam X/2\tau$.
\begin{enumerate}
\item \label{it-1-beyond} 
 If\/ $0 < q < p$ and $q\in\lQ$, then
  \begin{equation}\label{eq:low bounds beyond borderline} 
  \cp(B_r,B_R)\simge \frac{\mu(B_r)}{r^{q}}R^{q-p}
    = \frac{\mu(B_r)}{r^{p}}   \Bigl(\frac r R\Bigr)^{p-q}.
 \end{equation}
\item \label{it-2-beyond} 
 If $q > p$ and $q\in\uQ$, then
  \begin{equation}\label{eq:up bounds beyond borderline} 
  \cp(B_r,B_R) \simge \frac{\mu(B_R)}{R^{q}} r^{q-p}  
    = \frac{\mu(B_R)}{R^{p}} \Bigl(\frac r R\Bigr)^{q-p}.
 \end{equation}
\end{enumerate}
\end{prop}
\medskip

Proposition~\ref{prop-sharp} and
Example~\ref{ex1-partb} show that this result is sharp,
while unweighted $\R^n$, with $p=n$, 
shows that we cannot allow for $q=p$ in general. 
Also note that if 
$q\in\interior\lQ$ (resp.\ $q\in\interior\uQ$),
then \eqref{eq:low bounds beyond borderline} (resp.\ 
\eqref{eq:up bounds beyond borderline}) can be written as
$\cp(B_r,B_R) \simge \cq(B_r,B_R) R^{q-p}$
(resp.\ $\cp(B_r,B_R) \simge \cq(B_r,B_R) r^{q-p}$).

\begin{proof}
Let $u$ be admissible for $\cp(B_r,B_R)$,
and let $B^k$ be the corresponding balls, with radii $r_k$, 
from Lemma~\ref{lem:chain estimate}.
In \ref{it-1-p=1} we proceed as in~\eqref{eq:common} and use~\eqref{eq-est-Bk-Br} 
to obtain
\[ 
1 \simle \biggl({\frac{r^{q}}{\mu(B_r)}\int_{B_{R}} g_u^p \,d\mu}\biggr)^{1/p} 
       \sum_{k=1}^{k_0+1} r_k^{1-q/p} 
\simle R^{1-q/p} \biggl({\frac{r^{q}}{\mu(B_r)}\int_{B_{R}} g_u^p \,d\mu}
   \biggr)^{1/p},
\] 
since the exponent in the geometric series is $1-q/p>0$. 
Taking infimum over all admissible $u$ yields 
\eqref{eq:low bounds beyond borderline}. 

In \ref{it-2-p=1} we instead use~\eqref{eq:common} and~\eqref{eq-est-Bk-BR} 
and that the geometric series is $\simle r^{1-q/p}$ in this case.
\end{proof}

For the borderline cases $p=\max\lQ$ or $p=\min\uQ$, 
\eqref{eq:low bounds beyond borderline}
or~\eqref{eq:up bounds beyond borderline} can be used with $q$ arbitrarily close to $p$,
but the following proposition gives better estimates involving
logarithmic terms. 
If $X$ supports a $p_0$-Poincar\'e inequality at $x$ for some $1\le p_0<p$, then
even better 
estimates in the borderline cases are obtained in 
Proposition~\ref{prop:low bounds borderline-new}.
Nevertheless, the estimates  
in Proposition~\ref{prop-p=1} 
are of particular interest when $p=1$, since 
the $1$-Poincar\'e inequality is the best possible.

\begin{prop} \label{prop-p=1}
Assume that $\mu$ is reverse-doubling at $x$ and
supports  a \p-Poincar\'e inequality at $x$.
Let\/ $0<2r\le R \le \diam X/2\tau$.
\begin{enumerate}
\item \label{it-1-p=1}
If $p\in \lQ$, then
  \begin{equation}\label{eq:low bounds p=1}
  \cp(B_r,B_R)\simge 
    \frac{\mu(B_r)}{r^{p}}        \biggl(\log\frac R r\biggr)^{-p}.
 \end{equation}
\item \label{it-2-p=1}
If $p\in \uQ$, then
  \begin{equation}\label{eq:up bounds p=1}
  \cp(B_r,B_R) \simge 
    \frac{\mu(B_R)}{R^{p}}  \biggl(\log\frac R r\biggr)^{-p}.
 \end{equation}
\end{enumerate}
\end{prop}
\medskip

In unweighted $\R$ it is well known that $\cone(B_r,B_R)=2$
for all $0<r<R$.
In this case the right-hand sides in
\eqref{eq:low bounds p=1} and \eqref{eq:up bounds p=1}
both reduce to $2 (\log(R/r))^{-1}$, showing
that these estimates are not optimal 
in this particular case.

\begin{proof}
Let $u$ be admissible for $\cp(B_r,B_R)$,
and let $B^k$ be the corresponding balls, with radii $r_k$,
from Lemma~\ref{lem:chain estimate}.
Then~\eqref{eq:common} with $q=p$ and~\eqref{eq-est-Bk-Br} 
yield
\[
1 \simle k_0 \biggl(\int_{B_{R}} g_u^p \,d\mu\biggr)^{1/p} 
        \frac{r}{\mu(B_r)^{1/p}}.
\]
Since $k_0\simeq \log(R/r)$, taking infimum over all admissible $u$ yields 
\eqref{eq:low bounds p=1}.

In \ref{it-2-p=1} we instead use~\eqref{eq:common} and~\eqref{eq-est-Bk-BR}.
\end{proof}

\section{Capacity estimates for borderline exponents}
\label{sect:borderline}

When the borderline exponents are attained, then
Propositions~\ref{prop:simple upper bound} 
and~\ref{prop:upper bounds borderline} show that for $p=\max\lQ$,
\begin{equation} \label{eq-upper-1}
\cp(B_r,B_R) \le \min \biggl\{ \frac{\mu(B_{r})}{r^{p}}, 
    \frac{\mu(B_{R})}{R^{p}}\biggl(\log\frac R r\biggr)^{1-p} \biggr\},
\end{equation}
while for 
$p=\max\uQ$,
\begin{equation} \label{eq-upper-2}
\cp(B_r,B_R) \le \min \biggl\{ \frac{\mu(B_{R})}{R^{p}}, 
    \frac{\mu(B_{r})}{r^{p}}\biggl(\log\frac R r\biggr)^{1-p} \biggr\}.
\end{equation}
In this section we provide corresponding lower bounds,
even though the estimates do not exactly meet,
as seen in Theorem~\ref{thm:borderline-intro}.
Nevertheless, Proposition~\ref{prop-sharp} and 
Examples~\ref{ex1-partb}--\ref{ex-log-weights-large} below show that
all these estimates (including both possibilities
for the upper bounds in~\eqref{eq-upper-1} and in~\eqref{eq-upper-2})
are in some sense optimal;
see also Remark~\ref{rmk-log-weights}.

The following result holds for all $p\in \lQ$ (resp.\ $p\in \uQ$), 
but because of Proposition~\ref{prop:low bounds} it is most useful 
in the limiting case $p=\max \lQ$ (resp.\ $p=\min \uQ$).
It improves upon Proposition~\ref{prop-p=1}
at the cost of requiring a better Poincar\'e inequality; 
see the discussion on different Poincar\'e inequalities after 
Definition~\ref{def-PI}. 

\begin{prop}\label{prop:low bounds borderline-new}
Let\/ $1<p<\infty$ and\/ $0<2r\le R \le\diam X/2\tau$, and
assume that  $\mu$ is reverse-doubling at $x$ and
supports a $\po$-Poincar\'e inequality at $x$ for some\/ 
$1\leq \po <p$. 
\begin{enumerate}
\item \label{it-a-border}
 If $p\in\lQ$, then
\begin{equation}\label{eq-it-border-a}
\cp(B_r,B_{R})\simge \frac{\mu(B_r)}{r^{p}}\biggl(\log\frac R r\biggr)^{1-p}.
\end{equation}

\item \label{it-b-border}
If $p\in\uQ$, then
\begin{equation}\label{eq-it-border-b}
\cp(B_r,B_{R})\simge \frac{\mu(B_{R})}{R^{p}}\biggl(\log\frac R r\biggr)^{1-p}.
\end{equation}
\end{enumerate}
\end{prop}

Examples~\ref{ex-log-weights} and
\ref{ex-log-weights-large} show that these estimates are 
sharp, while
Proposition~\ref{prop-sharp} and 
Examples~\ref{ex1-partb}--\ref{ex-log-weights-large}
show that they 
do not hold for $p$ outside of the $Q$-sets.
In particular, these lower bounds
do not in general hold
for $p = \sup \lQ \notin \lQ$ and $p=\inf \uQ \notin \uQ$, respectively.

\begin{proof}
Let $u$ be admissible for $\cp(B_r,B_R)$,
and let $B^k$ be the corresponding balls, with radii $r_k$,
from Lemma~\ref{lem:chain estimate}.
Also let $A_k = \lambda B^{k}\setminus \lambda B^{k-1}$.

Without loss of generality we may assume that $\po>1$.
Lemma~\ref{lem:chain estimate} (with exponent $\po$) and
H\"older's inequality for sums (with $\po$ and $\po/(\po-1)$)
yield
\begin{align}\label{eq:first holder}
1 &\simle \sum_{k=1}^{k_0+1} r_k \biggl(\vint_{\lambda B^k} g_u^{\po} \,d\mu\biggr)^{1/\po}
  \le \sum_{k=1}^{k_0+1} \biggl(\frac{r_k^{p_0}}{\mu(B^k)}  
                \int_{\la B^k} g_u^{\po} \,d\mu\biggr)^{1/{\po}}   \nonumber\\
  &\le \bigl(k_0+1\bigr)^{1-1/{\po}}
          \biggl(\sum_{k=1}^{k_0+1} \frac{r_k^{p_0}}{\mu(B^k)}  
          \sum_{j=1}^{k} \int_{A_j} g_u^{\po} \,d\mu\biggr)^{1/{\po}}.
\end{align}
Interchanging the order of summation, the
double sum in~\eqref{eq:first holder} can be estimated by 
H\"older's inequality for integrals (with exponents $p/{\po}$ and $p/(p-{\po})$) as
\begin{align}\label{eq:second holder} 
\sum_{k=1}^{k_0+1}  \frac{r_k^{p_0}}{\mu(B^k)}  
          \sum_{j=1}^{k} \int_{A_j} g_u^{\po} \,d\mu
   &= \sum_{j=1}^{k_0+1} \int_{A_j} g_u^{\po} \,d\mu \sum_{k=j}^{k_0+1} 
         \frac{r_k^{p_0}}{\mu(B^k)} \nonumber\\
  & \simle \sum_{j=1}^{k_0+1} \biggl(\int_{A_j} g_u^{p} \,d\mu\biggr)^{\po/p}
    \mu(A_j)^{1-\po/p} \sum_{k=j}^{k_0+1} \frac{r_k^{p_0}}{\mu(B^k)}.
\end{align}
Let us now take $q\in\lQ$. 
(In~\ref{it-a-border} we can use $q=p$, 
but recall that also in~\ref{it-b-border} we have $\lQ\neq\emptyset$
by the reverse-doubling.)
Then 
\begin{equation}\label{eq:annulus estimate}
\mu(A_j)\simle \mu(B^j)\simle \mu(B^k)\biggl(\frac{r_j}{r_k}\biggr)^q
\end{equation} 
for $1\le j \le k \le k_0+1$.
Moreover, let 
$\ro=r$ if $p\in\lQ$ (case~\ref{it-a-border}) and $\ro=R$ if $p\in\uQ$
(case~\ref{it-b-border}). Then we have for all $1\le k \le k_0+1$ that
\begin{equation} \label{eq-p-in-Q}
\frac{r_k^p}{\mu(B^k)}\simle \frac{\ro^p}{\mu(B_\ro)}.
\end{equation}
From~\eqref{eq:annulus estimate} and~\eqref{eq-p-in-Q}
we obtain 
\begin{align}\label{eq:here was the difference}
\mu(A_j)^{1-\po/p} \sum_{k=j}^{k_0+1} \frac{r_k^{p_0}}{\mu(B^k)}
 \simle \sum_{k=j}^{k_0+1} \biggl( \frac{r_k^p}{\mu(B^k)} \biggr)^{\po/p}
        \biggl( \frac{r_j}{r_k} \biggr)^{q(1-\po/p)}
   \simle \biggl( \frac{\ro^{p}}{\mu(B_\ro)} \biggr)^{\po/p},
\end{align}
since $1-\po/p>0$, and thus $\sum_{k=j}^{k_0+1} (r_j/r_k)^{q(1-\po/p)}\simeq1$.

Insertion of~\eqref{eq:here was the difference}
into~\eqref{eq:second holder}, and a use of H\"older's 
inequality for sums (with exponents $p/{\po}$ and $p/(p-{\po})$), yields
\begin{align}\label{eq:final holder}
\sum_{k=1}^{k_0+1} \frac{r_k^{p_0}}{\mu(B^k)}  
          \sum_{j=1}^{k} \int_{A_j} g_u^{\po} \,d\mu
&\simle \sum_{j=1}^{k_0+1} \biggl(\int_{A_j} g_u^{p} \,d\mu\biggr)^{\po/p}
  \biggl( \frac{\ro^{p}}{\mu(B_\ro)} \biggr)^{\po/p} \\
& \le \biggl( \frac{\ro^p}{\mu(B_\ro)} \biggr)^{\po/p} 
      \biggl(\sum_{j=1}^{k_0+1} \int_{A_j} g_u^{p} \,d\mu\biggr)^{\po/p}\bigl({k_0+1}\bigr)^{1-\po/p}.  \nonumber  
\end{align}
Since $0<2r\le R$, we have
$k_0+1\simeq k_0\simeq\log(R/r)$, 
and so we conclude 
from~\eqref{eq:first holder} and~\eqref{eq:final holder}
that 
\begin{align}\label{eq:conclude}
1 & \simle k_0^{1-1/\po} \biggl( \biggl(\frac{\ro^p}{\mu(B_\ro)} 
         \int_{B_R} g_u^{p} \,d\mu\biggr)^{\po/p} 
      k_0^{1-\po/p} \biggr)^{1/\po} \nonumber \\ 
  & \simle\biggl(\log\frac R r\biggr)^{1-1/p} \biggl(\frac{\ro^p}{\mu(B_\ro)}  
         \int_{B_R} g_u^{p} \,d\mu\biggr)^{1/p}.
\end{align}
The desired capacity estimates~\eqref{eq-it-border-a} and~\eqref{eq-it-border-b} 
now follow from~\eqref{eq:conclude} 
by taking the infima over all $u$ admissible 
for $\cp(B_r,B_R)$
and recalling that $\ro=r$ in the case~\ref{it-a-border} 
and $\ro=R$ in the case~\ref{it-b-border}.
\end{proof}

\begin{proof}[Proof of Theorem~\ref{thm:borderline-intro}]
Combining Propositions~\ref{prop:low bounds borderline-new} 
and~\ref{prop:upper bounds borderline},
and appealing to Remark~\ref{rmk-Q-vs-Q0}
yield \ref{it3-a-thm-intro} and \ref{it3-b-thm-intro}.
The last part follows from Proposition~\ref{prop-sharp} below.
\end{proof}

\section{Capacity estimates involving \texorpdfstring{$S$}{S}-sets}
\label{sect-S}

Let us first record the following upper bounds
related to the $S$-sets.
As before, these upper 
estimates do not require any Poincar\'e inequalities.
Recall from Section~\ref{sect-sets-of-exponents} that the inequalities
defining the $S_\infty$-sets are reversed from the ones in the $S_0$-sets,
so that $\lQi \subset \lSi$ and $\uQi \subset \uSi$.

\begin{prop}\label{prop:upper with S}
Fix\/ $0<R_0<\infty$. 
\begin{enumerate}
\item \label{it-1-upS} 
If\/ $0<q\in \lSo$, then for\/ $0<2r \le R \le R_0$,
\begin{equation}  \label{eq-S-set-est-upper}
\cp(B_r,B_{R})\simle \begin{cases}
    R^{q-p}, & \text{if } q<p,\\
    r^{q-p}, & \text{if } q>p. \end{cases}
\end{equation}

\item \label{it-2-upS}
If\/ $0<q \in\uSi$, then \eqref{eq-S-set-est-upper} holds 
for $R_0 \le r \le R/2<\infty$.

\item \label{it-3-upS} 
If $p \in \lSo$, then for\/ $0<2r \le R \le R_0$,
\begin{equation}\label{eq-it-3-upS}
\cp(B_r,B_{R})\simle \biggl(\log\frac R r\biggr)^{1-p}.
\end{equation}

\item \label{it-4-upS}
If $p \in\uSi$, then~\eqref{eq-it-3-upS} holds
for $R_0 \le r \le R/2<\infty$.
\end{enumerate}
\end{prop}

In unweighted $\R^n$, the capacity $\cp(B_r,B_R)$
is comparable to the right-hand sides in the respective cases (with $q=n$), which
shows that these estimates are sharp.
See also the end of Example~\ref{ex1-partb}, where the sharpness of
part~\ref{it-1-upS} is shown in a case where $q\in\lSo\setminus\lQo$.

\begin{proof}
 The proofs of~\ref{it-1-upS} and~\ref{it-2-upS} follow immediately form
 Proposition~\ref{prop:simple upper bound} and the definitions of the $S$-sets.
 To see that~\ref{it-3-upS} and~\ref{it-4-upS} hold, one can proceed as in the
 proof of Proposition~\ref{prop:upper bounds borderline}
up to deducing \eqref{eq:up border}.
Then one uses the estimates
 $\mu(B^k)\simle r_k^p$ 
and $k_0+1 \simle \log(R/r)$ to
obtain \eqref{eq-it-3-upS}.
\end{proof}

The estimate \ref{it-3-upS} was already given by 
Heinonen~\cite{heinonen}, Lemma~7.18.
(The statement therein is slightly different, but the proof applies verbatim
to yield our estimate.)
It follows immediately from \ref{it-3-upS} that for $1<p\in\lSo$ the point $x$ has zero capacity,
but in fact the same is true even in the (possibly) larger set 
$[1,\infty)\setminus\uSo$, 
as the following proposition 
shows. Similarly, it follows from \ref{it-4-upS} 
that if $p \in\uSi$, then for a fixed $r>0$
we have $\cp(B_r,B_R)\to 0$ as $R\to\infty$, but again we obtain a better result in
Proposition~\ref{prop:S-infty}. 
Recall that 
\[
\inf \uSo  =\limsup_{r \to 0} \frac{\log \mu(B_r)}{\log r}
\quad \text{and} \quad
\inf \lSi  =\limsup_{r \to \infty} \frac{\log \mu(B_r)}{\log r}.
\]
by Lemma~\ref{lem-3} (and its $\infty$-version).

\begin{prop}\label{prop:zero cap}
If\/ $1\le p\notin \uSo$ or\/ $1 < p \in \lSo$,
then $\Cp(\{x\})=0=\cp(\{x\},B)$ for any ball
$B\ni x$.

Conversely, assume that
$\mu$ is reverse-doubling at $x$ and
supports a \p-Poincar\'e inequality at $x$, and that
there is a locally compact neigbourhood $G \ni x$.
If $p\in \interior \uSo$, then 
$\Cp(\{x\})>0$ and\/ $\cp(\{x\},B)>0$ for any ball $B\ni x$
with $\Cp(X \setm B)>0$.
\end{prop}

The first part of Proposition~\ref{prop:zero cap} 
improves and clarifies the result of
Corollary~3.4 in Garofalo--Marola~\cite{GaMa}.
Note that this part
is valid without requiring any Poincar\'e inequality.
Unweighted $\R$ shows that the inequality in  $1 < p \in \lSo$
is necessary. 
The second part, on the other hand, is a consequence of 
Proposition~\ref{prop-Cp=0<=>cp=0} and the lower bound in
Proposition~\ref{prop:S and capacity}
 below.

In the remaining 
case when $p = \min \uSo$ and $p \notin \lSo$,
the $S$-sets are not enough to determine
if the capacities of $\{x\}$ are zero or not,
as we demonstrate at the end
of Example~\ref{ex-log-weights}.

\begin{prop}\label{prop:S and capacity}
Assume that $\mu$ is reverse-doubling at $x$ and supports 
a \p-Poincar\'e inequality at $x$,
and fix\/ $0<R_0<\infty$.
Furthermore, assume that\/ $0<q\in \uSo$ and\/ $0<r<R\le R_0$,
or that $q \in\lSi$ and $R_0 \le r<R<\infty$.
Then
\begin{equation*}
\cp(B_r,B_{R})\simge \begin{cases}
    R^{q-p}, & \text{if } q<p,\\
    r^{q-p}, & \text{if } q>p, \\
    {\displaystyle \biggl(\log\frac R r\biggr)^{-p}}, & \text{if } q=p.
\end{cases}
\end{equation*}
\end{prop}

Also here unweighted $\R^n$ shows that the first two estimates are
sharp, and at the end of Example~\ref{ex1-partb} their sharpness
is shown in a case where $q\in\uSo\setminus\uQo$. 
Proposition~\ref{prop-sharp-S} provides a converse of 
Proposition~\ref{prop:S and capacity}.

\begin{proof}
Let $u$ be admissible for $\cp(B_r,B_R)$,
and let $B^k$ be the corresponding balls, with radii $r_k$,
from Lemma~\ref{lem:chain estimate}.
Since 
$\mu(\lambda B^k) \ge \mu(B^k)\simge r_k^q$ for all $k\in\N$,
we have by Lemma~\ref {lem:chain estimate} that
\begin{align} \label{eq:when in uSo-new}
1 & \simle \sum_{k=1}^{k_0+1} r_k \biggl( \vint_{\la B^k} g_u^p\,d\mu \biggr)^{1/p}
\nonumber\\
  &\simle \sum_{k=1}^{k_0+1} r_k^{1-q/p} \biggl( \int_{B_R} g_u^p\,d\mu \biggr)^{1/p} 
\simeq A \biggl( \int_{B_R} g_u^p\,d\mu \biggr)^{1/p},
\end{align}
where 
\[
    A= \begin{cases}
     R^{1-q/p}, & \text{if } q<p, \\
     r^{1-q/p}, & \text{if } q>p, \\
     k_0+1,   & \text{if } q=p.
      \end{cases}
\]
The claim then follows by taking infimum over all admissible $u$.
\end{proof}

\begin{proof}[Proof of Proposition~\ref{prop:zero cap}]
We may assume that $B=B_{R}$.
If $p\notin \uSo$, then there exist $r_n\to0$ such that 
$\mu(B(x,r_n))<r_n^p/n$.
For $r_n\le R$, 
let $u_n(y)=(1-d(x,y)/r_n)_\limplus$.
Then 
$u_n(x)=1$, $u_n=0$ outside $B(x,r_n)$, and $g_{u_n}\le1/r_n$.
Thus
\[ 
\cp(\{x\},B) \le \|g_{u_n}\|^p_{L^p(X)} \le \|u_n\|^p_{\Np(X)} 
\]
and 
\[
\Cp(\{x\}) \le \|u_n\|^p_{\Np(X)} 
\le (1+r_n^{-p}) \mu(B(x,r_n)) 
\simle \frac1n \to 0, \quad \text{as } n\to\infty.
\] 

In the case $1 < p \in \lSo$ the claim $\cp(\{x\},B)=0$
follows easily from
Proposition~\ref{prop:upper with S}\,\ref{it-3-upS}. 
To show that also $\Cp(\{x\})=0$, we let $\eps>0$.
Since $p \in \lSo$ we can find $r>0$ such that $\mu(B(x,r))< \eps$.
As $\cp(\{x\},B(x,r))=0$ 
(by Proposition~\ref{prop:upper with S}\,\ref{it-3-upS} again), we can also find
$u \in \Np_0(B(x,r))$ such that $u(x)=1$, $0 \le u \le 1$
and $\int_X g_u^p \, d\mu < \eps$.
It follows that 
\[
   \Cp(\{x\}) \le \|u\|_{\Np(X)}^p 
   \le \mu(B(x,r)) + \int_X g_u^p \, d\mu 
   < 2\eps \to 0,
   \quad \text{as } \eps \to 0.
\]

Conversely, assume 
that $p>q\in\uSo(x)$.
By Proposition~\ref{prop:S and capacity}
 we have for all $0<r<R=:R_0$ that
\begin{equation} \label{eq-cp=0-proof}
\cp(B_r,B_{R})\simge {R^{q-p}},
\end{equation}
with comparison constant independent of $r$.
If $\cp(\{x\},B)$ were $0$, then we would have $\Cp(\{x\})=0$,
by Proposition~\ref{prop-Cp=0<=>cp=0},
which in turn, by Proposition~\ref{prop-zero-cap-capacitable},
would contradict \eqref{eq-cp=0-proof}.
Hence $\cp(\{x\},B)>0$ and $\Cp(\{x\})>0$.
\end{proof}

\begin{remark}
It follows directly from Proposition~\ref{prop:S and capacity}
that if we \emph{a priori} know that the capacity is outer
or that the capacity of singletons can be tested
by only continuous functions,
then actually $\cp(\{x\},B_R) \simge R^{q-p}$ whenever $p>q\in\uSo$.
Both of the above assumptions hold e.g.\ if $X$ is complete,
$\mu$ is doubling and supporting a \p-Poincar\'e inequality, by
Theorem~6.19 in~\cite{BBbook} or
Kallunki--Shanmugalingam~\cite{KaSh},
see also Theorem~4.1 in Bj\"orn--Bj\"orn~\cite{BBvarcap}.
\end{remark}

Let us also record the following logarithmic 
lower bound, which improves the third lower bound in
Proposition~\ref{prop:S and capacity} and
is interesting in the borderline cases $p=\min\uSo$
and $p=\max\lSi$.

\begin{prop}\label{prop:S and capacity borderline}
Let\/ $1<p<\infty$ and
assume that 
$\mu$ is reverse-doubling at $x$ and
supports 
a $\po$-Poincar\'e inequality at $x$ for some\/ 
$1\leq \po <p$. 
Fix\/ $0<R_0<\infty$.
\begin{enumerate}
\item \label{it-3-S} 
If $p \in \uSo$ and\/ $0<2r \le R \le R_0$, then
\begin{equation}\label{eq-it-3-S}
\cp(B_r,B_{R})\simge \biggl(\log\frac R r\biggr)^{1-p}.
\end{equation}

\item \label{it-4-S}
If $p \in\lSi$ and $R_0 \le r \le R/2<\infty$, then~\eqref{eq-it-3-S} holds.
\end{enumerate}
\end{prop}

\begin{proof}
We proceed as in the proof
of Proposition~\ref{prop:low bounds borderline-new}, but instead 
of~\eqref{eq-p-in-Q} we now have the simple estimate 
$r_k^p/\mu(B^k)\simle 1$ for all $1\le k \le k_0+1$
(both in \ref{it-3-S} and~\ref{it-4-S}). 
Thus the left-hand side 
of~\eqref{eq:here was the difference} is bounded by a 
constant. 
Inserting this into \eqref{eq:second holder} and
then \eqref{eq:first holder},
together with a use of H\"older's inequality for
sums as in \eqref{eq:final holder},
 yields
\[
1 \simle\biggl(\log\frac R r\biggr)^{1-1/p} \biggl(
         \int_{B_R} g_u^{p} \,d\mu\biggr)^{1/p},
\]
since $k_0+1 \simeq \log(R/r)$.
Taking infimum over all admissible $u$ yields~\ref{it-3-S} and~\ref{it-4-S}.
\end{proof}

In unbounded spaces we have the following counterpart to 
Proposition~\ref{prop:zero cap}.
Recall that the sets $\lSi$ and $\uSi$ are independent
of the reference point 
$x\in X$, by Lemma~\ref{lem:infty sets invariant}.

\begin{prop} \label{prop:S-infty}
Assume that $X$ is unbounded.
If\/ $1 \le p\notin \lSi$ or\/ $1 <p  \in \uSi$, then\/ $\cp(B_r,X)=0$ for all 
$r>0$, and thus\/ $\cp(E,X)=0$ for all bounded sets~$E$.

Conversely, assume that
$\mu$ is reverse-doubling at $x$ and supports 
a \p-Poincar\'e inequality at $x$. If
$p\in \interior \lSi$, then
\[  
\cp(B_r,B_R)\ge \cp(B_r,X) \ge c(r)>0
\quad \text{for all\/ } 0<r<R.
\]
\end{prop}

Unweighted $\R$ again shows that the inequality in  $1 < p \in \uSi$
is necessary.
In the remaining 
case when $p = \max \lSi$ and $p \notin \uSi$,
the $S$-sets are not enough to determine
if the capacities are zero or not,
see the end
of Example~\ref{ex-log-weights-large}.

\begin{proof}
If $p\notin \lSi$, then there exist $R_n\to\infty$ such that $\mu(B_{R_n})< R_n^p/n$.
By Proposition~\ref{prop:simple upper bound} we have
\[
\cp(B_r,X)\le\cp(B_r,B_{R_n})\simle \frac{\mu(B_{R_n})}{R_n^p} < \frac1n \to 0, 
  \quad \text{as } n\to\infty.
\]
If  $1 <p  \in \uSi$ we instead
use Proposition~\ref{prop:upper with S}\,\ref{it-4-upS}
to conclude that $\cp(B_r,X)=0$.

Conversely, if $p<q\in\lSi$, then let ${R_0}:=r<R$. 
From Proposition~\ref{prop:S and capacity}
 we
obtain that
\[
\cp(B_r,B_{R})\simge {r^{q-p}},
\]
and the claim follows from Lemma~\ref{lem-cpX}.
\end{proof}

\begin{remark}
Recall that an unbounded proper space $X$ is said to be \emph{\p-parabolic},
if $\cp(K,X)=0$ for all compact sets $K\subset X$, and otherwise
$X$ is \emph{\p-hyperbolic}. 
From Proposition~\ref{prop:S-infty} it thus follows
that the space $X$ is \emph{\p-parabolic} if
$1\le p\notin\lSi$
(or $1 <p  \in \uSi$), and $X$ is \emph{\p-hyperbolic} if
$p\in \interior \lSi$.
See e.g.\ Holopainen~\cite{Ho}, Holopainen--Koskela~\cite{HoKo}
and Holopainen--Shanmugalingam~\cite{HoSha} 
for more information on parabolic and hyperbolic
Riemannian manifolds and metric spaces.
\end{remark}

\section{Sharpness of the estimates}
\label{sect-sharpness}

The following result shows that the lower bounds in Sections~\ref{sect:interior} 
and~\ref{sect:borderline} are not only sharp, but also essentially equivalent to $p$ (or $q$)
belonging to the corresponding $Q$-sets.

\begin{prop} \label{prop-sharp}
If~\eqref{eq-it-1}, \eqref{eq-it-2}, 
\eqref{eq:low bounds p=1}, \eqref{eq:up bounds p=1}, 
\eqref{eq-it-border-a} 
or~\eqref{eq-it-border-b} holds for all\/ $0<2r\leq R$, then
$p\in\lQ$, $p\in\uQ$, 
$p\le\sup\lQ$, $p\ge\inf\uQ$,
$p\le\sup\lQ$ or $p\ge\inf\uQ$, respectively.

Similarly, if~\eqref{eq:low bounds beyond borderline} or \eqref{eq:up bounds beyond borderline}
holds for all\/ $0<2r\leq R$, then $q\in\lQ$ or $q\in\uQ$, respectively.  
\end{prop}

\begin{proof}
We need to estimate $\mu(B_r)/\mu(B_R)$ in terms of $r/R$ for all $0<r<R$.
It is enough to do this for $0<2r\leq R$, since $R/2<r<R$ can be 
treated by the doubling property of $\mu$ at $x$.
If~\eqref{eq-it-1} or~\eqref{eq-it-2} holds, then 
Proposition~\ref{prop:simple upper bound} yields
\[
\frac{\mu(B_r)}{r^{p}} \simle \cp(B_r,B_R) \simle \frac{\mu(B_{R})}{R^{p}}
\quad \text{or} \quad
\frac{\mu(B_R)}{R^{p}} \simle \cp(B_r,B_R) \simle \frac{\mu(B_{r})}{r^{p}},
\]
which is equivalent to $p\in\lQ$ or $p\in\uQ$, respectively.

Next, if \eqref{eq:low bounds beyond borderline}
holds for some $q>0$, then using 
Proposition~\ref{prop:simple upper bound} we see that
\[
   \frac{\mu(B_r)}{r^{q}}R^{q-p} 
   \simle \cp(B_r,B_R)\simle  \frac{\mu(B_R)}{R^p},
\]
which after division by $R^{q-p}$ shows that $q \in \lQ$.
Similarly, if \eqref{eq:up bounds beyond borderline}
holds for some $q>0$, then $q \in \uQ$.

Finally, if~\eqref{eq:low bounds p=1} holds,
and in particular if~\eqref{eq-it-border-a} holds, then 
Proposition~\ref{prop:simple upper bound} yields for all $\eps>0$
that 
\[
\frac{\mu(B_r)}{\mu(B_R)} 
\simle \frac{r^p \cp(B_r,B_R)}{\mu(B_R)} 
{ \log^p \frac{R}{r}}   
\simle \Bigl( \frac{r}{R} \Bigr)^p 
{\log^p \frac{R}{r}}  
\simle \Bigl( \frac{r}{R} \Bigr)^{p-\eps},
\]
where the last implicit constant depends on $\eps$.
Thus $p-\eps\in\lQ$ for every $\eps>0$, showing that $p\le\sup\lQ$.
The implications 
\eqref{eq-it-border-b} $\imp$ \eqref{eq:up bounds p=1} $\imp$
$p\ge\inf\uQ$ are proved similarly.
\end{proof}

We have a corresponding result for the $S$-sets as well.

\begin{prop}\label{prop-sharp-S}
If for some $q>0$ and all\/ $0<2r\leq R \le R_0$,
\begin{equation}  \label{eq-lower-S-for-sharp}
\cp(B_r,B_R) \simge r^{q-p} \quad \text{or} \quad
\cp(B_r,B_R) \simge R^{q-p},
\end{equation}
then $q\in \uSo$. Similarly, if
\begin{equation}  \label{eq-lower-S-for-sharp-log}
\cp(B_r,B_R) \simge \biggl( \log \frac{R}{r} \biggr)^{-p}
\end{equation}
for all\/ $0<2r\leq R \le R_0$, then $p\ge\inf\uSo$.

If instead \eqref{eq-lower-S-for-sharp} or \eqref{eq-lower-S-for-sharp-log} 
holds for all $R_0\le r\le R/2<\infty$, then $q\in\lSi$ or $p\le\sup\lSi$,
respectively.
\end{prop}

\begin{proof}
We prove only the case $0<2r\leq R \le R_0$, the other case being similar.

If \eqref{eq-lower-S-for-sharp} holds and $q\le p$, then 
Proposition~\ref{prop:simple upper bound} implies that 
$R^{q-p} \simle \cp(B_r,B_R) \simle R^{-p} \mu(B_R)$ for all $R\le R_0$, 
showing that $q\in\uSo$.
If
instead $q\ge p$ and \eqref{eq-lower-S-for-sharp} holds, then we get
$r^{q-p} \simle \cp(B_r,B_R) \simle r^{-p} \mu(B_r)$ for all $r\le R_0/2$,
and the same conclusion follows.

If \eqref{eq-lower-S-for-sharp-log} holds, then 
Proposition~\ref{prop:simple upper bound} and taking $R=R_0$ show that
$\log^{-p} (R_0/r) \simle \cp(B_r,B_R) \simle r^{-p} \mu(B_r)$.
Since $\log (R_0/r) \simle r^{-\eps}$ for every $\eps>0$, this yields
$\mu(B_r)\simge r^{p(1+\eps)}$, and hence $p(1+\eps)\in\uSo$.
Letting $\eps\to0$, gives $p\ge \inf\uSo$.
\end{proof}

In the rest of this section we continue
our study of the examples from Section~\ref{sect-ex}, 
using
a general formula for the capacity 
on weighted $\R^n$ with radial weights.
The proof of this formula is postponed until
Section~\ref{sect:weights}, see
Proposition~\ref{prop-cap-formula-radial}.

\begin{example} \label{ex1-partb}
We  
continue with our Example~\ref{ex1}.
First, for $p >2$ and $2\alp_{k+1}\le2r\le R\le \be_k$, we estimate using 
Proposition~\ref{prop-cap-formula-radial} with 
$f'(\rho)\simeq w(\rho)\rho$
and~\eqref{eq-ex1-4} that
\begin{align}  \label{eq-est-cap-al-be}
 \cp(B_r, B_R) 
    & \simeq \biggl( \int_{r}^{R} (\alp_{k+1} \rho)^{1/(1-p)}\,d\rho \biggr)^{1-p} \\
    & \simeq \alp_{k+1} ( R^{(p-2)/(p-1)} - r^{(p-2)/(p-1)})^{1-p} 
     \simeq \alp_{k+1} R^{2-p} \simeq \frac{\mu(B_R)}{R^p}, \nonumber
\end{align}
showing that the second upper bound in 
Proposition~\ref{prop:simple upper bound} cannot be improved.
With $r=\al_{k+1}$ and $R=\be_k$ it also follows that 
\begin{equation}\label{eq:cap-div-to-zero}
   \frac{ \cp(B_{\alp_{k+1}}, B_{\be_k})}{\alp_{k+1}^{-p} \mu(B_{\alp_{k+1}})}
    \simeq \biggl( \frac{\alp_{k+1}}{\be_k} \biggr)^{p-2} 
    = \alp_k^{p/2-1} 
    \to 0, 
    \quad \text{as } k \to \infty,
\end{equation}
since $p>2$.
This illustrates the fact (known from Proposition~\ref{prop-sharp}) 
that the lower estimates 
\eqref{eq-it-1}, \eqref{eq:low bounds p=1}
and \eqref{eq-it-border-a} do not
hold for $p >2$, i.e.\ for $p \notin \lQ$.
In addition, the equivalence in~\eqref{eq:cap-div-to-zero} shows 
that the lower bound in
\eqref{eq:low bounds beyond borderline}
is sharp (with $q=2\in\lQ$).
As
\[
    \biggl( \log \frac{\be_k}{\alp_{k+1}} \biggr)^{1-p} 
    = ( \log \alp_{k}^{-1/2} )^{1-p} 
    \to 0, 
    \quad \text{as } k \to \infty,
\]
we also conclude from~\eqref{eq-est-cap-al-be} (with $r=\al_{k+1}$ and $R=\be_k$)
that the estimate \eqref{eq-it3-a} does not
hold for $p >2$, i.e.\ for $p \notin \lQ$.

If $1 <p <4$ and $2\be_k\le2r\le R \le \al_{k}$, then by 
Proposition~\ref{prop-cap-formula-radial} with 
$f'(\rho)\simeq w(\rho)\rho$
and~\eqref{eq-Br-le-al},
\begin{align} \label{eq-est-cap-be-al}
 \cp(B_{r}, B_{R}) 
    & \simeq \biggl( \int_{r}^{R} \biggl(\frac{\rho^2}{\alp_{k}} 
        \rho\biggr)^{1/(1-p)}\,d\rho \biggr)^{1-p} \\
    & \simeq \frac{1}{\alp_k} ( R^{(p-4)/(p-1)} - r^{(p-4)/(p-1)})^{1-p} 
     \simeq \frac{r^{4-p}}{\alp_{k}} \simeq \frac{\mu(B_r)}{r^p}, \nonumber
\end{align}
showing that the first upper bound in 
Proposition~\ref{prop:simple upper bound} cannot be improved.
In particular, \eqref{eq-est-cap-al-be} and \eqref{eq-est-cap-be-al} show that
each of the upper bounds in Proposition~\ref{prop:simple upper bound}
can give a sharp estimate for certain radii even when
$p\notin\lQ\cup\uQ$.

With $r=\be_{k}$ and $R=\al_k$ it 
follows from~\eqref{eq-est-cap-be-al} that
\[
   \frac{ \cp(B_{\be_{k}}, B_{\alp_k})}{\alp_{k}^{-p} \mu(B_{\alp_{k}})}
    \simeq \biggl( \frac{\be_k}{\alp_k} \biggr)^{4-p}
    = \alp_k^{2-p/2} 
    \to 0, 
    \quad \text{as } k \to \infty,
\]
since $p<4$. 
Thus we here have a concrete case where 
the lower estimates \eqref{eq-it-2}, \eqref{eq:up bounds p=1}
and \eqref{eq-it-border-b} do not
hold for $p <4$, i.e.\ for $p \notin \uQ$, and we also see that 
\eqref{eq:up bounds beyond borderline}
is sharp as well
(with $q=4\in\uQ$).
Moreover, as 
\[
    \biggl( \log \frac{\al_k}{\be_{k}} \biggr)^{1-p} 
    = ( \log \alp_{k}^{-1/2} )^{1-p} 
    \to 0, 
    \quad \text{as } k \to \infty,
\]
we conclude from~\eqref{eq-est-cap-be-al} (with $r=\be_{k}$ and $R=\al_k$)
that the estimate \eqref{eq-it3-b} does not
hold for $1 <p<4$, i.e.\ for $p \notin \uQ$.

From \eqref{eq-est-cap-be-al} and~\eqref{eq-est-cap-al-be} with $p=2$ and $p=4$,
respectively, we see that
\[
 \ctwo(B_{\be_{k}}, B_{\alp_k}) 
\simeq \frac{\mu(B_{\be_k})}{\be_k^2} 
\quad \text{and} \quad
 \capp_4(B_{\al_{k+1}}, B_{\be_k}) \simeq \frac{\mu(B_{\be_k})}{\be_k^4},
\]
which shows that the lower bounds in~\eqref{eq-it-border-a} and~\eqref{eq-it-border-b}
are not always comparable to $\cp(B_r,B_R)$ when $p=\max \lQ$ or 
$p=\min \uQ$, and that the estimates provided by
Proposition~\ref{prop:simple upper bound} are in this case optimal 
(and better than those in Proposition~\ref{prop:upper bounds borderline}).

Finally, choosing $R=\be_k$ and $p>q=\tfrac{10}{3}=\min\uSo$ 
in \eqref{eq-est-cap-al-be}
(or $r=\be_k$ and $1 <p<q=\tfrac{10}{3}=\min\uSo$ in 
\eqref{eq-est-cap-be-al})
shows, together with \eqref{eq-ex1-3b}, that 
the first two lower bounds in Proposition~\ref{prop:S and capacity}
are sharp.
Similarly for $p<q=3=\max\lSo$, we see from \eqref{eq-ex1-3} and 
\eqref{eq-est-cap-be-al} with $r=\tfrac12\al_k$ 
and $R=\al_k$ that the upper
bounds in Proposition~\ref{prop:upper with S}\,\ref{it-1-upS} are sharp.
\end{example}

\begin{example}  \label{ex-log-weights}
This is a continuation of Example~\ref{ex-log-weight-Q} in $\R^n$, 
$n\ge2$, with the weight
\[
      w(\rho)=\begin{cases}
          \rho^{p-n} \log^\be (1/\rho), & \text{if } 0<\rho\le1/e, \\
          \rho^{p-n} , & \text{otherwise},
        \end{cases}
\]
where 
$\be\in\R$ is arbitrary and we this time require $p>1$.
Recall that for $0<r<1/e$ and $x=0$ we have $\mu(B_r)\simeq r^p \log^\be (1/r)$.
Proposition~\ref{prop-cap-formula-radial} with $f'(\rho)\simeq w(\rho)\rho^{n-1}$ gives, for 
$0<r<R<1/e$, that
\begin{align}  \label{eq-log-cap-first}
\cp(B_r,B_R)^{1/(1-p)}  
&\simeq \int_r^R \log^{\be/(1-p)} (1/\rho) \,\frac{d\rho}{\rho}
= \int_{\log(1/R)}^{\log(1/r)}  t^{\be/(1-p)} \,dt \nonumber\\
&= \frac{1}{\s} \Bigl( \log^\s \frac1r - \log^\s \frac1R \Bigr)
\end{align}
if $\s=1+\be/(1-p)\ne0$, and 
\[
\cp(B_r,B_R)^{1/(1-p)} \simeq \log \log \frac1r - \log \log \frac1R
\]
if $\be=p-1$. 

The estimate~\eqref{eq-log-cap-first} can be further simplified.
For that we recall the simple Lemma~3.1 from 
Bj\"orn--Bj\"orn--Gill--Shanmugalingam~\cite{tree} which says 
that for all
$\s>0$ and all $t\in[0,1]$,
\[
\min\{1,\s\} t \le 1-(1-t)^\s \le \max\{1,\s\}t.
\]
Thus, if $\s>0$ in~\eqref{eq-log-cap-first}, we have
\begin{align}  \label{eq-log-cap-use-lemma}
\frac{1}{\s} \Bigl( \log^\s \frac1r - \log^\s \frac1R \Bigr)
&\simeq {\biggl(}\log^\s \frac1r{\biggr)} \biggl( 1 - 
     \biggl(\frac{\log (1/R)}{\log (1/r)} \biggr)^\s \biggr) \\
&\simeq {\biggl(}\log^\s \frac1r {\biggr)}\biggl( 1 - \frac{\log (1/R)}{\log (1/r)} \biggr)
= \biggl( \log \frac1r \biggr)^{\s-1} \log \frac Rr. \nonumber
\end{align}
Since $\s-1=\be/(1-p)$, this together with~\eqref{eq-log-cap-first} gives
\begin{align}  \label{eq-log-cap-s>0}
\cp(B_r,B_R) 
 &\simeq {\biggl( \log^\be \frac1r \biggr)} \biggl( \log \frac Rr \biggr)^{1-p} 
 {\simeq} \frac{\mu(B_r)}{r^p} \biggl( \log \frac Rr \biggr)^{1-p}.
\end{align}
On the other hand, if $\s<0$ in~\eqref{eq-log-cap-first} then replacing $\s$ 
by $\theta=-\s>0$ in~\eqref{eq-log-cap-use-lemma} yields
\begin{align*}  
\frac{1}{\s} \Bigl( \log^\s \frac1r - \log^\s \frac1R \Bigr)
&\simeq \biggl(\log^\s \frac1r \biggr) \biggl(\log^\s \frac1R \biggr) 
    \biggl( \log^\theta \frac1r - \log^\theta \frac1R \biggr) \\
&\simeq \biggl(\log^\s \frac1r \biggr) \biggl(\log^\s \frac1R \biggr) 
        \biggl( \log \frac1r \biggr)^{\theta-1} \log \frac Rr
= \frac{\log^\s (1/R)}{\log (1/r)} \log \frac Rr.
\end{align*}
Since $\s(1-p)=\be(1-(p-1)/\be)$, we obtain 
from \eqref{eq-log-cap-first} that
\begin{align}  \label{eq-log-cap-s<0}
\cp(B_r,B_R) 
& \simeq \biggl( \frac{\log^\s (1/R)}{\log (1/r)} \biggr)^{1-p} 
                 \biggl( \log \frac Rr \biggr)^{1-p} \nonumber\\
& = \biggl( \log^\be \frac1R \biggr)^{1-(p-1)/\be} 
          \biggl( \log^\be \frac1r \biggr)^{(p-1)/\be} 
                \biggl( \log \frac Rr \biggr)^{1-p} \nonumber\\
& \simeq \biggl( \frac{\mu(B_R)}{R^p} \biggr)^{1-(p-1)/\be} 
          \biggl( \frac{\mu(B_r)}{r^p} \biggr)^{(p-1)/\be} 
                \biggl( \log \frac Rr \biggr)^{1-p}.
\end{align}
We now distinguish three cases.

(a) If $\be<0$, then $p=\max\lQ$ and $\s>0$.
Thus \eqref{eq-log-cap-s>0} 
yields
\[
\cp(B_r,B_R) \simeq \frac{\mu(B_r)}{r^p} \biggl( \log \frac Rr \biggr)^{1-p}.
\]
This shows that the lower estimate in 
Proposition~\ref{prop:low bounds borderline-new}\,\ref{it-a-border} is sharp 
and
that \eqref{eq-it-border-b} fails in this case,
despite the fact that $p = \inf \uQ$.

(b) If $0<\be<p-1$, then $p=\min\uQ$ and $\s>0$.
From~\eqref{eq-log-cap-s>0} 
we conclude that
\[
\cp(B_r,B_R) \simeq \frac{\mu(B_r)}{r^p} \biggl( \log \frac Rr \biggr)^{1-p},
\]
this time showing that the upper estimate in 
Proposition~\ref{prop:upper bounds borderline}\,\ref{it3-b} is sharp.

(c) If $\be>p-1$, then $p=\min\uQ$ and $\s<0$.
From~\eqref{eq-log-cap-s<0} we see that
\begin{equation}   \label{eq-cap-est-both-R-r}
\cp(B_r,B_R) \simeq \biggl( \frac{\mu(B_R)}{R^p} \biggr)^{1-(p-1)/\be} 
          \biggl( \frac{\mu(B_r)}{r^p} \biggr)^{(p-1)/\be} 
                \biggl( \log \frac Rr \biggr)^{1-p}.
\end{equation}
Note that both exponents $1-(p-1)/\be$ and $(p-1)/\be$ are positive and
their sum is 1. 
Letting $\be\to\infty$ and $\be\to p-1$, respectively,
shows that in general for $p=\min\uQ$ the estimate 
\[
\frac{\mu(B_R)}{R^p} \biggl( \log \frac Rr \biggr)^{1-p} \simle
\cp(B_r,B_R) \simle \frac{\mu(B_r)}{r^p} \biggl( \log \frac Rr \biggr)^{1-p},
\]
is the best we can hope for, since the definitions of $\lQ$ and $\uQ$
cannot capture the size of $\be$ in $\mu(B_\rho)\simeq\rho^p\log^\be (1/\rho)$,
only its sign.
Thus also the lower estimate in 
Proposition~\ref{prop:low bounds borderline-new}\,\ref{it-b-border} 
is optimal. 
 
In addition, 
if $R$ is fixed and $r <R$, then by \eqref{eq-log-cap-s<0},
\[
\cp(B_r,B_R)  \simeq
\biggl( \log \frac 1 r \biggr)^{p-1} \biggl( \log \frac Rr \biggr)^{1-p}.
\]
When $r \ll R$, this is
substantially smaller  (since $\be>p-1$) 
than the lower bound 
\[
\frac{\mu(B_r)}{r^p}\biggl( \log \frac Rr \biggr)^{1-p}
\simeq\biggl( \log \frac 1 r \biggr)^{\beta}\biggl( \log \frac Rr \biggr)^{1-p}
\] 
claimed 
in~\cite[Theorem~3.2]{GaMa} for the case $p=q(x)=\sup\lQ$. 
Thus the latter estimate cannot be valid
if $p=\sup\lQ=\min\uQ \notin\lQ$. 
Similarly, for $\pt>p=\sup\lQ =\min\uQ \notin\lQ$ 
we have by Proposition~\ref{prop:simple upper bound}
that
\[
\cpt(B_r,B_R) \simle \frac{\mu(B_R)}{R^\pt}
\simeq \biggl( \log \frac 1 R \biggr)^{\beta} R^{p-\pt}.
\]
For $\be>0$ and $r\ll R$, this is again substantially smaller than 
 \[
\frac{\mu(B_r)}{r^p} R^{p-\pt} 
   \simeq \biggl( \log \frac 1 r \biggr)^{\beta} R^{p-\pt},
\]
showing that the lower bound claimed  
in~\cite[Theorem~3.2]{GaMa} 
for the case $\pt>q(x)$ cannot be valid in general. 
Nevertheless, let us point out that if $q(x)=\max\lQ$,
then the estimates given in \cite[Theorem~3.2]{GaMa}
for the cases $\pt=q({x})$ and $\pt>q({x})$
are (essentially) the same as our 
Propositions~\ref{prop:low bounds borderline-new}\,\ref{it-a-border}
and~\ref{prop:low bounds beyond borderline}\,\ref{it-1-beyond},
respectively.

We now turn to the $S$-sets.
If $\be>0$, then $\lSo=\lQ=(0,p)$ and $\uSo=\uQ=[p,\infty)$.
Thus, Proposition~\ref{prop:zero cap} is of no use,
and indeed we can 
show that both 
$\Cp(\{0\})=0$ and $\Cp(\{0\})>0$
are possible in this case:

If $\sigma <0$, i.e.\ if $\be >p-1$,
then 
$
      \lim_{r\to0} \cp(B_r,B_R) >0,
$
by \eqref{eq-log-cap-first}.
In the same way as at the end of the proof of Proposition~\ref{prop:zero cap}
it follows that $\Cp(\{0\})>0$ and $\cp(\{0\},B)>0$
for every ball $B \ni 0$.

If instead $\sigma >0$, i.e.\ if $0 < \be < p-1$,
then 
$
      \lim_{r\to0} \cp(B_r,B_R) =0,
$
by \eqref{eq-log-cap-first}, from which it
directly follows that $\cp(\{0\},B)=0$
for every ball $B \ni 0$.
Using that $\Cp(\{0\}) \le \cp(\{0\},B) + \mu(B)$
shows that also $\Cp(\{0\})=0$.
\end{example}

\begin{example}  \label{ex-log-weights-large}
Let  
\[  
    w(\rho)=\begin{cases}
   \rho^{p-n} \log^\be \rho & \text{for } \rho\ge e, \\
   \rho^{p-n}, & \text{otherwise},
   \end{cases}
\]
in $\R^n$, $n \ge 2$,
where $p>1$ and  $\be\in\R$ is arbitrary, as 
in the second part of Example~\ref{ex-log-weight-Q}. 
This example is similar to the previous 
example, 
but the roles of $r$ and $R$ are in a sense reversed 
and thus we obtain
different estimates. 

As in Example~\ref{ex-log-weights}, we have $\sup\lQ=\inf\uQ={p}$, 
but 
if $\be>0$ it is now $\sup\lQ$ that is attained, while for $\be<0$
we have that $\inf\uQ$ is attained.
Since
\[
f'(\rho) \simeq w(\rho)\rho^{n-1} = \rho^{p-1}\log^\be \rho
 \quad \text{for }\rho>e,
\]
we have by Proposition~\ref{prop-cap-formula-radial} for 
$e<r<R$ the estimate 
\begin{align}  \label{eq-log-cap-first-large}
\cp(B_r,B_R)^{1/(1-p)}  
& \simeq \int_r^R \log^{\be/(1-p)} {(\rho)} \,\frac{d\rho}{\rho}
  \nonumber \\
&
= \int_{\log r}^{\log R}  t^{\be/(1-p)} \,dt 
= \frac{\log^\s R - \log^\s r}{\s}
\end{align}
if $\s=1+\be/(1-p)\ne0$.

The simplification of~\eqref{eq-log-cap-first-large} can be carried out analogously
to the previous example, and we obtain for $\s>0$ that
\begin{align}  \label{eq-log-cap-s>0-large}
\cp(B_r,B_R) &\simeq \Bigl( \log^\s R - \log^\s r \Bigr)^{1-p}
\nonumber\\
 &\simeq \biggl( \log R \biggr)^\be \biggl( \log \frac Rr \biggr)^{1-p} 
 \simeq \frac{\mu(B_R)}{R^p} \biggl( \log \frac Rr \biggr)^{1-p}.
\end{align}
This yields in the cases corresponding to (a) and (b) of Example~\ref{ex-log-weights}
the following conclusions:

(a) If $\be<0$, then $p=\min\uQ$ and $\s>0$.
Thus~\eqref{eq-log-cap-s>0-large} 
shows the sharpness of the lower estimate in 
Proposition~\ref{prop:low bounds borderline-new}\,\ref{it-b-border}.
It also shows that \eqref{eq-it-border-a} fails in this case,
despite the fact that $p = \sup \lQ$.

(b) If $0<\be<p-1$, then $p=\max\lQ$ and $\s>0$,
and from~\eqref{eq-log-cap-s>0-large} we can conclude 
that also the upper estimate in 
Proposition~\ref{prop:upper bounds borderline}\,\ref{it3-a} is sharp.

We also mention that the case $\s<0$ can be studied
just as in Example~\ref{ex-log-weights}\,(c), this time showing the
sharpness of the lower bound in 
Proposition~\ref{prop:low bounds borderline-new}\,\ref{it-a-border},
although this was already known from the case (a) of 
Example~\ref{ex-log-weights}; see however Remark~\ref{rmk-log-weights} below.

Finally, if $\be >0$, then $\lSi=\lQ =(0,p]$ and
$\uSi=\uQ=(p,\infty)$, and thus 
Proposition~\ref{prop:S-infty} is of no use. 
Considering the two cases
$\sigma>0$ and $\sigma<0$ shows that indeed both possibilities
$\cp(B_r,X)=0$ and $\cp(B_r,X)>0$ can happen in this case,
cf.\ the end of Example~\ref{ex-log-weights}.
\end{example}

\begin{remark}\label{rmk-log-weights}
In Example~\ref{ex-log-weights} we have $\lQ=\lQo$ and $\uQ=\uQo$,
and thus the conclusions of this example also show the sharpness of
the respective restricted capacity estimates, that is, the analogues
of Proposition~\ref{prop:upper bounds borderline}\,\ref{it3-b} and
Proposition~\ref{prop:low bounds borderline-new}\,\ref{it-a-border} and~\ref{it-b-border}
for $\lQo$ and $\uQo$ and for radii $0<2r\le R \le R_0$.
In particular, Theorem~\ref{thm:borderline-intro},
with the exception of the upper bound in~\eqref{eq-it3-a-thm-intro},
is shown to be sharp.

Similarly,
in Example~\ref{ex-log-weights-large} we have $\lQ=\lQi$ and $\uQ=\uQi$,
and so we obtain the sharpness of the analogues
of Proposition~\ref{prop:upper bounds borderline}\,\ref{it3-a} and
Proposition~\ref{prop:low bounds borderline-new}\,\ref{it-a-border} and~\ref{it-b-border}
for $\lQi$ and $\uQi$ and for radii $R_0\le r\le R/2$.

Nevertheless, these examples still leave open the sharpness of one of the
upper bounds in each of
the restricted versions of Proposition~\ref{prop:upper bounds borderline}: We
do not know if the upper estimate~\eqref{eq-it3-a}
is sharp for $p\in \lQo$ and $0<2r\le R\le R_0$, 
or if~\eqref{eq-it3-b}
is sharp for $p\in \uQi$ and $R_0\le r\le R/2$.
\end{remark}

\section{Radial weights and stretchings in \texorpdfstring{$\R^n$}{Rn}}
\label{sect:weights}

In this section we consider radial weights in $\R^n$, $n\ge2$, and give
a sufficient condition for when they are admissible, and in particular satisfy
the global doubling condition and a global 
Poincar\'e inequality, 
thus 
providing a basis for our examples in Section~\ref{sect-sharpness}.
This will be achieved by comparing such weights with
suitable powers of Jacobians of quasiconformal mappings on $\R^n$.
In particular, in Theorem~\ref{thm-quasiconf}
we characterize those radial stretchings in $\R^n$ which are quasiconformal.
The same condition was considered in $\R^2$ by 
Astala--Iwaniec--Martin~\cite[Section~2.6]{AsIwMa} 
and for continuously differentiable mappings in $\R^n$ by 
Manojlovi\'c~\cite[Example~2.9]{Manojlovic},
while for power-like radial stretchings the corresponding result is well known, see
e.g.\ Example~16.2 in V\"ais\"al\"a~\cite{vaisala}.
Both in~\cite{AsIwMa} and~\cite{Manojlovic}, the result is obtained by differentiation and
uses the analytical definition of quasiconformal mappings, based on the Jacobian determinant.
Our assumptions are weaker and the method is different 
and based on more direct estimates of the linear dilation, 
rather than on the differentiable structure of $\R^n$.
We use
the following metric definition of quasiconformal mappings, provided by e.g.\ 
Theorem~34.1 in~\cite{vaisala}, and applicable also in metric spaces.

\begin{deff}
A homeomorphism $F\colon\R^n\to\R^n$, $n \ge 2$,
is a \emph {quasiconformal mapping} if its 
\emph{linear dilation}
\[
H_F(x):= \limsup_{r\to0} \frac{L(x,r)}{l(x,r)}
\]
is bounded. Here
\[
L(x,r):= \max_{|x-y|=r} |F(x)-F(y)| \quad \text{and} \quad
l(x,r):= \min_{|x-y|=r} |F(x)-F(y)|.
\]
\end{deff}

We shall consider radial stretchings $F\colon\R^n\to\R^n$ given by 
\begin{equation}   \label{eq-def-F}
F(x)=h(|x|)x=k(|x|)\frac{x}{|x|} \quad \text{if } x\ne0, \quad \text{and} \quad F(0)=0,
\end{equation}
where $h(\rho)=k(\rho)/\rho$,
and $k$ is a 
locally absolutely continuous 
homeomorphism of $[0,\infty)$
satisfying $k(0)=0$ and
\begin{equation}  \label{eq-k'-k/rho-est}
m \le \frac{\rho k'(\rho)}{k(\rho)} \le M
\end{equation}
for a.e.\ $\rho\in[0,\infty)$ and some $0<m\le M<\infty$.
It is easily verified that the inverse mapping of $F$ is given by
\[
F^{-1}(z)=k^{-1}(|z|)\frac{z}{|z|},
\] where the inverse $k^{-1}$ is (under our assumptions) also locally
absolutely continuous, and by~\eqref{eq-k'-k/rho-est} we have for a.e.\ 
$\rho\in [0,\infty)$ that
\begin{equation}   \label{eq-inverse-derivative}
(k^{-1}(\rho))' = \frac{1}{k'(k^{-1}(\rho))} \simeq \frac{k^{-1}(\rho)}{k(k^{-1}(\rho))}
= \frac{k^{-1}(\rho)}{\rho},
\end{equation}
where the implicit constants in $\simeq$ are $1/M$ and $1/m$.

We are going to obtain the following characterization.

\begin{thm} \label{thm-quasiconf}
Assume that the mapping $F\colon\R^n\to\R^n$, $n \ge 2$,
is defined as in~\eqref{eq-def-F}.
Then  $F$ is quasiconformal if and only if 
\eqref{eq-k'-k/rho-est} holds 
for a.e.\ $\rho\in[0,\infty)$ and some\/ $0 <m \le M < \infty$.
\end{thm}

The following lemma gives a basis for 
the sufficiency part of the theorem. 

\begin{lem}  \label{lem-Lip-const}
If $F\colon\R^n\to\R^n$ is as in~\eqref{eq-def-F} and satisfies~\eqref{eq-k'-k/rho-est},
then for all $x,y\in\R^n$, with\/ $|x|\le|y|$ and $x \ne y$, we have
\begin{equation}   \label{eq-Fx-Fy-h}
\frac{m}{1+2m} \inf_{|x|\le\xi\le|y|} h(\xi) \le \frac{|F(x)-F(y)|}{|x-y|} 
\le (M+2) \sup_{|x|\le\xi\le|y|} h(\xi).
\end{equation}
\end{lem}

\begin{proof}
For $x=0$ this is easily checked using the definition of $F$, so assume for the
rest of the proof that $x \ne 0$.
The triangle inequality yields
\begin{align}   \label{eq-triangle-ineq}
|F(x)-F(y)| &= \bigl|h(|x|)x - h(|y|)x + h(|y|)x - h(|y|)y\bigr| \nonumber\\
&\le h(|y|) |x-y| + |x|\, \bigl|h(|y|)-h(|x|)\bigr|.
\end{align}
Note that $h$ is also locally absolutely continuous and the 
assumption~\eqref{eq-k'-k/rho-est} gives for a.e.\ $|x|\le\xi\le|y|$,
\[
|h'(\xi)|=\biggl| \frac{k'(\xi)}{\xi}-\frac{k(\xi)}{\xi^2} \biggr| 
\le \frac{(M+1) k(\xi)}{\xi^2} \le (M+1) \frac{h(\xi)}{|x|}.
\]
Hence
\[ 
\bigl|h(|y|)-h(|x|)\bigr|\le (|y|-|x|) \esssup_{|x|\le\xi\le|y|} |h'(\xi)|
\le (M+1) |x-y| \sup_{|x|\le\xi\le|y|} \frac{h(\xi)}{|x|}.
\] 
Inserting this into~\eqref{eq-triangle-ineq} proves the 
second inequality in \eqref{eq-Fx-Fy-h}.

To prove the first inequality we use the inverse mapping $F^{-1}(z)=k^{-1}(|z|)z/|z|$.
By~\eqref{eq-inverse-derivative}, it satisfies~\eqref{eq-k'-k/rho-est}
with $m$ and $M$ replaced by $1/M$ and $1/m$.
The first part of the proof applied to $F^{-1}$ with $z=F(x)$ and 
$w=F(y)$ then yields
\[
\frac{|x-y|}{|F(x)-F(y)|} = \frac{|F^{-1}(z)-F^{-1}(w)|}{|z-w|}
\le \biggl( \frac{1}{m}+2 \biggr) \sup_{|z|\le\z\le|w|} \frac{k^{-1}(\z)}{\z}. 
\]
Since $k^{-1}(\z)/\z = \xi/k(\xi) = 1/h(\xi)$ with $\xi=k^{-1}(\z)$, 
the first inequality in \eqref{eq-Fx-Fy-h} follows.
\end{proof}

\begin{proof}[Proof of Theorem~\ref{thm-quasiconf}.]
First assume that \eqref{eq-k'-k/rho-est} holds.
If $x=0$, then $L(x,r)=l(x,r)$ by the definition of $F$, and so $H_F(0)=1$.
If on the other hand
$x \ne 0$, then by Lemma~\ref{lem-Lip-const} and the definition of $F$ we have,
for $0<r<|x|$,
\begin{equation}   \label{eq-est-L-l-h}
L(x,r) \simle r  \sup_{|x|-r\le\xi\le|x|+r} h(\xi)  \quad \text{and} \quad 
l(x,r) \simge r  \inf_{|x|-r\le\xi\le|x|+r} h(\xi).
\end{equation}
Inserting this into the definition of $H_F(x)$ and letting $r\to0$ 
shows that $F$ is quasiconformal.

Conversely, assume that $F$ is quasiconformal.
Since the linear dilation $H_F(x)$ is bounded,
Theorem~32.1 in V\"ai\-s\"a\-l\"a~\cite{vaisala} 
shows that $F$ is differentiable a.e.
It follows that $k'$ exists a.e.\ in $(0,\infty)$.
To prove~\eqref{eq-k'-k/rho-est}, choose $K>0$ such that $H_F < K$ in $\R^n$.
Fix $x\in\R^n$ with $|x|=1$ and let $\rho>0$ be arbitrary but such that $k'(\rho)$ exists.
Then there exists $0 < r_0 < \rho$ such that 
\(
L(\rho x,r) \le K l(\rho x,r) 
\)
whenever $0<r\le r_0$.
For each such $r$ find $y\in\R^n$ such that $|y|=1$ and $|x-y|=r/\rho$.
Then $|\rho x -\rho y|=r$ and
\[
l(\rho x,r) \le |F(\rho x) - F(\rho y)| = k(\rho)|x-y| = \frac{k(\rho) r}{\rho}.
\]
On the other hand, 
\begin{align*}
\frac{k(\rho+r)-k(\rho)}{r} & = 
\frac{|F((\rho+r)x)-F(\rho x)|}{r}
\le \frac{L(\rho x,r)}{r} \le K \frac{l(\rho x,r)}{r}
\le K \frac{k(\rho)}{\rho},
\end{align*}
and the quotient $(k(\rho)-k(\rho-r))/r$ can be treated similarly.
Letting $r\to0$ shows that $k'(\rho)\le K k(\rho)/\rho$.
Applying the same argument to the quasiconformal mapping $F^{-1}$
yields, with $\z=k(\rho)$,
\[
\frac{1}{k'(\rho)} = (k^{-1}(\z))' \le K \frac{k^{-1}(\z)}{\z} = \frac{K\rho}{k(\rho)},
\]
i.e.\ $k'(\rho)\ge k(\rho)/K\rho$.
\end{proof}

Now assume that $F$ is as in Lemma~\ref{lem-Lip-const}.
The Jacobian $J_F$ of $F$ is the infinitesimal area distortion under
$F$, and thus~\eqref{eq-est-L-l-h} 
implies that $J_F(x)\simeq h(|x|)^n$
for a.e.\ $x\in\R^n$.
Since 
Jacobians of quasiconformal mappings
are strong $A_\infty$ weights (by a result due to Gehring~\cite{Gehring},
cf.\ pp.\ 101--102 in David--Semmes~\cite{DaSe} and Theorem~1.5 in 
Heinonen--Koskela~\cite{HeKo-Scand}), 
Theorem~1 in J.~Bj\"orn~\cite{JB-Fenn} shows that the weight
\[
J_F(x)^{1-p/n}\simeq h(|x|)^{n-p}=\biggl( \frac{k(|x|)}{|x|} \biggr)^{n-p}
\]
is \p-admissible when $1\le p\le n$. 
(For $1<p\le n$, one can instead use Theorem~15.33 in 
Heinonen--Kilpel\"ainen--Martio~\cite{HeKiMa} or Corollary~1.10 in 
Heinonen--Koskela~\cite{HeKo-Scand}.)
We thus have the following result.

\begin{thm}  \label{thm-cf-admiss}
Let $k\colon[0,\infty)\to[0,\infty)$ be a 
locally absolutely continuous homeomorphism of\/ $[0,\infty)$
satisfying~\eqref{eq-k'-k/rho-est} for a.e.\ $\rho\in[0,\infty)$.
Then the weight $w(x)=(k(|x|)/|x|)^{n-p}$ with\/ $1\le p\le n$ 
is \p-admissible in\/ $\R^n$, $n\ge2$.
\end{thm}

Now let $w$ be a radial weight on $\R^n$, $n\ge2$, i.e.\ $w(x)=w(|x|)$ where
$0\le w\in L^1\loc(0,\infty)$.
Here we abuse the notation and use $w$ both for the weight itself and for its 
one-dimensional representation on $(0,\infty)$.
With the help of Theorem~\ref{thm-cf-admiss} we obtain the following 
sufficient condition for admissibility of radial weights.

\begin{prop}   \label{prop-suff-admiss-w}
Assume that $w\colon(0,\infty)\to(0,\infty)$ is locally absolutely continuous and that 
for some $\ga_1<n-1$, $\ga_2<\infty$ and a.e.\ $\rho>0$ we have,
\begin{equation}   \label{eq-cond-w'/w}
-\ga_1 \le \frac{\rho w'(\rho)}{w(\rho)} \le \ga_2.
\end{equation}
Then the radial weight $w(x)=w(|x|)$ is\/ $1$-admissible in\/ $\R^n$, $n\ge2$.
\end{prop}

\begin{remark} \label{rmk-cor-w}
In particular, Proposition~\ref{prop-suff-admiss-w}
shows that all the weights 
\[
w(x)=\begin{cases}
          |x|^\al \log^\be(1/{|x|}), & \text{if } 0<|x|\le1/e, \\
          |x|^\al, & \text{otherwise},
        \end{cases}
\]
with $\al>1-n$ and $\be\in\R$, are 1-admissible in $\R^n$, $n\ge2$.
We expect these weights to be 1-admissible (and even $A_1$)
for $-n<\al \le 1-n$ as well, 
but the $A_1$ condition needs to be checked in this case.
This is well known for $\be=0$, 
see Heinonen--Kilpel\"ainen--Martio~\cite[p.\ 10]{HeKiMa},
thus showing that the above condition for
admissibility is not sharp.
Note also that for $n=1$ a weight is \p-admissible if and only if it is
an $A_p$ weight, by Theorem~2 in Bj\"orn--Buckley--Keith~\cite{BBK}, and that
the above ``Jacobian'' technique does not apply in this case.
\end{remark}

\begin{proof}[Proof of Proposition~\ref{prop-suff-admiss-w}]
Let $k(\rho)=\rho w(\rho)^{1/(n-1)}$.
Then $k$ is locally absolutely continuous and \eqref{eq-cond-w'/w}
implies that
\begin{align}  \label{eq-est-k'-ga1}
k'(\rho) &= w(\rho)^{1/(n-1)} + \frac{1}{n-1} \rho w(\rho)^{1/(n-1)-1} w'(\rho) \\
&= w(\rho)^{1/(n-1)} \biggl( 1 + \frac{\rho w'(\rho)}{(n-1) w(\rho)} \biggr) 
\ge \biggl( 1 - \frac{\ga_1}{n-1} \biggr) w(\rho)^{1/(n-1)}, \nonumber
\end{align}
which is positive for a.e.\ $\rho$.
Thus $k$ is strictly increasing.
Note also that integrating the inequality 
$w'(\rho)/w(\rho)\ge-\ga_1/\rho$ implies that 
\[
\frac{w(\rho_2)}{w(\rho_1)} \ge \biggl( \frac{\rho_2}{\rho_1} \biggr)^{-\ga_1}
\]
for $0<\rho_1\le\rho_2<\infty$, and hence
\begin{align*}
k(\rho_2) &= \rho_2 w(\rho_2)^{1/(n-1)} \simge \rho_2^{1-\ga_1/(n-1)} \to \infty,
\quad \text{as }\rho_2\to\infty,  
\intertext{and}
k(\rho_1) &= \rho_1 w(\rho_1)^{1/(n-1)} \simle \rho_1^{1-\ga_1/(n-1)} \to 0,
\quad \text{as }\rho_1\to0,
\end{align*}
showing that $k$ is onto.
From~\eqref{eq-cond-w'/w} and \eqref{eq-est-k'-ga1} we also conclude that
\[
0<1-\frac{\ga_1}{n-1} \le \frac{\rho k'(\rho)}{k(\rho)} 
  \le 1+\frac{\ga_2}{n-1},
\]
i.e.\ that~\eqref{eq-k'-k/rho-est} holds.
Theorem~\ref{thm-cf-admiss} now finishes the proof.
\end{proof}

\begin{remark}
The condition~\eqref{eq-cond-w'/w} can also be expressed in terms of
$f(\rho):=\mu(B(0,\rho))$, where $d\mu=w\,dx$, as follows. 
Since $w(\rho)=C\rho^{1-n}f'(\rho)$, an equivalent condition 
to~\eqref{eq-cond-w'/w} is 
\[
0<n-1-\ga_1 \le \frac{\rho f''(\rho)}{f'(\rho)} \le n-1+\ga_2.
\]
Note that this requires $f''>0$ (since $f$ is increasing), i.e.\ $f$ must be convex,
which excludes small powers $f(r)= r^\al$, $0<\al<1$.
On the other hand, these correspond to $A_1$ weights, and are thus
1-admissible;
see Heinonen--Kilpel\"ainen--Martio~\cite[p.~10]{HeKiMa}
and Theorem~4 in J.~Bj\"orn~\cite{JB-Fenn}, and cf.\ also 
Remark~\ref{rmk-cor-w}.
\end{remark}

We end this section by calculating the variational capacity of annuli with respect
to radial weights in $\R^n$. 

\begin{prop} \label{prop-cap-formula-radial}
Let $w(x)=w(|x|)$ be a radial weight on\/ $\R^n$, $n\ge2$, such that  $w>0$ a.e.\ 
and $w\in L^1\loc(\R^n)$. 
Assume that the corresponding measure $d\mu = w\,dx$ supports
a \p-Poincar\'e inequality at\/ $0$, where $p>1$.
Let $f(r)=\mu(B_r)$, where $B_r=B(0,r)\subset\R^n$.
Then
\[
\cp(B_r,B_R) = \biggl( \int_r^R (f')^{1/(1-p)} \,d\rho \biggr)^{1-p}
\quad \text{whenever\/ } 0 < r < R \le \infty.
\]
\end{prop}

In Section~\ref{sect-sharpness} we applied this formula 
to various weights including weights of logarithmic type.
In Theorems~2.18 and~2.19 in
Heinonen--Kilpel\"ainen--Martio~\cite{HeKiMa}, an integral
estimate was obtained for nonradial weights satisfying the $A_p$ condition.
See also Theorem~3.1 in Holopainen--Koskela~\cite{HoKo}, 
where capacity of annuli in
Riemannian manifolds is estimated in a similar way.

\begin{remark}
For Proposition~\ref{prop-cap-formula-radial},   
we actually do not need the full
\p-Poincar\'e inequality at $0$, it is enough to have it for some ball
$B \supset B_R$ with $\mu(B \setm B_R)>0$.
The Poincar\'e inequality is only used
when proving Lemma~\ref{lem-PI=>abs-cont-radial}, which in turn
is used to show that the minimizer $u$ for $\cp(B_r,B_R)$
is absolutely continuous on
rays and that $g_u=|u'|$. 

These consequences are not always
true if the Poincar\'e assumption is omitted. 
Indeed, if e.g.\  
\[
w(\rho)=\rho^{1-n} \biggl( \sum_{j=1}^\infty 
2^{-j} \biggl( 1 + \frac{1}{|\rho-q_j|} \biggl) \biggl)^{-p} \le \rho^{1-n},
\]
where $\{q_j\}_{j=1}^\infty$ is an enumeration of the positive rational
numbers, then
\(
g(x)=\sum_{j=1}^\infty 
2^{-j} \bigl( 1 + \bigl||x|-q_j\bigr|^{-1} \bigr) \in L^p(B_R,w\,dx)
\)
is for every $\rt\in (r,R)$ an upper gradient of
$u:=\chi_{B_{\rt}}$, 
since $\int_\gamma g\,ds=\infty$ for every curve $\gamma$ 
crossing over $\bdry B_{\rt}$.
Thus $u\in\Np_0(B_R, w\,dx)$ and Corollary~2.21 in
Bj\"orn--Bj\"orn~\cite{BBbook}  
implies that $g_u=0$ a.e.\ in $B_R$. 
It follows that the minimizer is not unique (and may also be nonradial)
and $\cp(B_r,B_R)=0$ in this case. 
Moreover, $g\in\Np(B_R,w\,dx)$ (with itself as an upper gradient),
but $g\notin L^1\loc(\R^n,dx)$, so $g'$ need not be defined (e.g.\ in
the distributional sense).
Cf.\ also the discussion after Proposition~\ref{prop-Cp=0<=>cp=0}
and the discussion about gradients on p.\ 13 in 
Heinonen--Kilpel\"ainen--Martio~\cite{HeKiMa}.

On the other hand, if $w$ is \p-admissible, then Theorem~8.6 in~\cite{HeKiMa}
directly shows that $\cp(B_r,B_R)=\int_{B_R\setm B_r} |\grad u|^p\,w\,dx$, 
where $u$ is the solution of 
\[
\Div(w(x)|\grad u(x)|^{p-2} \grad u(x))=0 
\quad \text{in } B_R\setm B_r
\]
 with the boundary data $1$ on $\bdry B_r$ and $0$ on 
$\bdry B_R$, and only the second half of the proof
below is needed in this case, cf.\ Example~2.22 in~\cite{HeKiMa}.
More general weights require more care and are treated using the metric
space theory.
\end{remark}

\begin{proof}[Proof of Proposition~\ref{prop-cap-formula-radial}]
By Lemma~\ref{lem-cpX} we may assume that $R < \infty$.
First we have
\[
f(r)=\int_{B_r}w\,dx = \omega_{n-1} \int_0^rw(\rho) \rho^{n-1}\,d\rho,
\]
where $\omega_{n-1}$ is the surface area of the 
$(n-1)$-dimensional unit sphere in $\R^n$.
To calculate $\cp(B_r,B_R)$ we need to minimize 
$\int_{B_R\setm B_r}  g_u^p w\,dx$ 
among functions $u$ with $u=1$ on $B_r$
and $u=0$ on $\bdry B_R$.
We shall also see below that under our assumptions, $g_u=|u'|$ a.e.

Since the data are bounded, no
Poincar\'e inequality nor doubling property is needed for the
existence of a minimizer (i.e.\ 
a competing function having \p-energy equal to $\cp(B_r,B_R)$),
by e.g.\ Theorem~5.13 in Bj\"orn--Bj\"orn~\cite{BBnonopen}.
Without such assumptions the minimizer need not
be unique and there may exist a nonradial minimizer,
but there always exists at least one radial minimizer.
Indeed, if $v$ is a minimizer, then
\[
\cp(B_r,B_R) =\int_{B_R\setm B_r} g_v^p w\,dx 
= \int_{\Sp^{n-1}}  \int_r^R g_v^p(\rho \theta) w(\rho)
\rho^{n-1}\,d\rho \, d\theta,
\]
and we can find $\theta_0 \in \Sp^{n-1}$ so that 
\begin{equation}\label{eq:one ray} 
  \int_r^R g_v^p(\rho \theta_0) w(\rho)
     \rho^{n-1}\,d\rho \le \frac{\cp(B_r,B_R)}{\om_{n-1}}.
\end{equation}
Letting $u(x)=v(|x|\theta_0)$ and $g(x)=g_v(|x|\theta_0)$
it is easily verified that $g$ is a \p-weak upper gradient
of $u$ and that, by~\eqref{eq:one ray},
\[
  \int_{B_R\setm B_r} g_u^p w\,dx 
  \le \int_{B_R\setm B_r} g^p w\,dx  \le 
  \cp(B_r,B_R).
\]
Thus $u$ is a radial minimizer.

As usual, we write $u(x)=u(|x|)$, where $u\colon[0,\infty)\to\R$.
We may clearly assume that ${u}$ 
is decreasing, and so 
$u'(\rho)$ exists for a.e.\ $\rho$. 
By Proposition~3.1 in Shanmugalingam~\cite{Sh-rev} (or Theorem~1.56 in~\cite{BBbook}), 
$u$ is absolutely continuous on all curves, except for a curve family with zero
\p-modulus (with respect to the measure $\mu$).
Lemma~\ref{lem-PI=>abs-cont-radial} below shows that the family of all radial rays 
connecting $B_r$ to $\R^n\setm B_R$ has positive \p-modulus.
By symmetry, it then follows that $u$ is absolutely continuous on the interval $[r,R]$ 
and hence, by Lemma~2.14 in Bj\"orn--Bj\"orn~\cite{BBbook}, $g_u = |u'|$ a.e.

Thus,
\[
\int_{B_R\setm B_r} g_u^p w\,dx 
= \omega_{n-1} \int_r^R |u'(\rho)|^p w(\rho) \rho^{n-1}\,d\rho 
= \int_r^R |u'(\rho)|^p f'(\rho)\,d\rho. 
\]
Since $u$ is a minimizer of this integral, it 
solves the corresponding Euler--Lagrange equation
\begin{equation*}   
(|u'|^{p-2}u'f')'=0
\end{equation*}
(which is derived in a standard way)
and hence $|u'|^{p-2}u'f'=A$ a.e. 
It is clear that
$u'\le0$, and so we get $u'=-(A/f')^{1/(p-1)}$ a.e.
To 
determine the constant $A$, notice that 
\begin{equation*}   
1=u(r)-u(R) =-\int_r^Ru'(\rho)\,d\rho = 
\int_r^R \biggl( \frac{A}{f'} \biggr)^{1/(p-1)}\,d\rho,
\end{equation*}
and thus 
\[
A=\biggl( \int_r^R (f')^{1/(1-p)}\,d\rho \biggr)^{1-p}.
\]
Inserting this into the above expressions for $u'$ and 
$\int_{B_R\setm B_r} g_u^p w\,dx$ gives
\begin{align} \label{eq-calc-3}
\cp(B_r,B_R) &= \int_r^R |u'|^p f'\,d\rho 
= \biggl( \int_r^R (f')^{1/(1-p)}\,d\rho \biggr)^{-p}
\int_r^R (f')^{p/(1-p)} f'\,d\rho \nonumber\\ 
& = \biggl( \int_r^R (f')^{1/(1-p)} \,d\rho \biggr)^{1-p}.
\qedhere
\end{align}
\end{proof}

\begin{lem}  \label{lem-PI=>abs-cont-radial}
Under the assumptions of Proposition~\ref{prop-cap-formula-radial}, 
the family\/ $\Ga_{r,R}$
of all radial rays connecting $B_r$ to\/ $\R^n\setm B_R$ 
has positive \p-modulus with respect
to the measure $d\mu = w\,dx$. 
\end{lem}

\begin{proof}
Assume on the contrary that the \p-modulus of $\Ga_{r,R}$ is zero. Then 
there exists $g\in L^p(B_R\setm B_r,\mu)$ such that for every radial
ray $\ga$ connecting $r\theta$ to $R\theta$, where $\theta\in \Sp^{n-1}$, we have 
\[
\int_\ga g\,ds = \int_r^R g(\rho\theta)\,d\rho = \infty.
\]
Since $g\in L^p(B_R\setm B_r,\mu)$,
Fubini's theorem implies that for a.e.\ $\theta\in \Sp^{n-1}$,
\[
\int_r^R g(\rho\theta)^p w(\rho\theta) \rho^{n-1} \,d\rho < \infty.
\]
Choose one such $\theta\in \Sp^{n-1}$ and set 
$\gt(|x|)=\gt(x)=g(|x|\theta)$, $x\in B_R\setm B_r$.
Then $\gt$ is radially symmetric, $\gt\in L^p(B_R\setm B_r,\mu)$,  and
$\int_\ga \gt\,ds = \infty$ for every $\ga\in\Ga_{r,R}$.

Since $\int_r^R \gt\,dt=\infty$, we can by successively halving
intervals find a decreasing  sequence of intervals $[a_j,b_j]$ such that
$\int_{a_j}^{b_j} \gt\,dt=\infty$ and $b_j - a_j \to 0$, as $j \to
\infty$. Letting $\rt=\lim_{j \to \infty} a_j$ we see that
either
$\int_{\rt - \eps}^{\rt} \gt\,dt=\infty$ for all $\eps >0$,
 or $\int_{\rt}^{\rt+\eps} \gt\,dt=\infty$ for all $\eps >0$
(or both).
Let in the former case $E=B_{\rt}$ and in the latter case
$E=\itoverline{B}_{\rt}$.

If $\ga\colon[0,l_\ga]\to\R^n$
 is any (nonradial) curve connecting $E$ to $\R^n\setm E$, then using
the symmetry of $\gt$ it is easily verified that 
\[
\int_\ga \gt\,ds \ge \int_{|\ga(0)|}^{|\ga(l_\ga)|}
\gt\, dt 
= \infty.
\]
Thus $\gt$ is an upper gradient of $u_n=n\chi_{E}$ for every $n=1,2,\ldots$\,.
Since $u_n\in\Np(B_{2R},\mu)$,
applying the \p-Poincar\'e
inequality at $0$ to $u_n$ gives
\[
0 < n \vint_{B_{2R}} |u_1 - u_{1,B_{2R}}|\,d\mu 
= \vint_{B_{2R}} |u_n - u_{n,B_{2R}}|\,d\mu 
\le CR \biggl( \vint_{B_{2R}} \gt^p \,d\mu \biggr)^{1/p} < \infty.
\]
Letting $n\to\infty$ leads to a contradiction, showing that $\Ga_{r,R}$ has positive
\p-modulus.
\end{proof}

\end{document}